\documentclass[11pt, reqno]{amsart}

\usepackage[top=3.75cm, bottom=3cm, left=3cm, right=3cm]{geometry}
\usepackage{amsmath}
\usepackage{amsthm}
\usepackage{amsfonts}
\usepackage{amssymb}
\usepackage{mathtools}
\usepackage{bm}
\usepackage[colorlinks=true, citecolor=citation, urlcolor=citation, linkcolor=reference]{hyperref}
\usepackage{color}
\usepackage[all,cmtip]{xy}
\usepackage{enumitem}
\usepackage{tikz}  
\usepackage{graphicx}
\usepackage{scalerel}
\usepackage{todonotes}
\usepackage{comment}
\usepackage[mathscr]{eucal}

\definecolor{citation}{rgb}{0,.40,.80}
\definecolor{reference}{rgb}{.80,0,.40}

\numberwithin{equation}{section}

\theoremstyle{plain}
\newtheorem{theorem}{Theorem}[section]
\newtheorem{lemma}[theorem]{Lemma}
\newtheorem{proposition}[theorem]{Proposition}
\newtheorem{corollary}[theorem]{Corollary}
\newtheorem{conjecture}[theorem]{Conjecture}
\newtheorem{question}[theorem]{Question}

\theoremstyle{definition}
\newtheorem{definition}[theorem]{Definition}
\newtheorem{example}[theorem]{Example}
\newtheorem{remark}[theorem]{Remark}

\newtheorem{warning}[theorem]{Warning}

\setlist[itemize]{leftmargin=*, itemsep={2pt}}
\setlist[enumerate]{leftmargin=*, itemsep={2pt}} 

\newcommand{\st}{\mid} 
\newcommand{\set}[1]{\left\{ \, #1 \, \right\}}

\newcommand{\Db}{\mathrm{D^b}}
\newcommand{\Dperf}{\mathrm{D}_{\mathrm{perf}}}
\newcommand{\Dqc}{\mathrm{D}_{\mathrm{qc}}}
\newcommand{\llangle}{\left \langle}
\newcommand{\rrangle}{\right \rangle}
\DeclareMathOperator{\Forg}{Forg}

\DeclareMathOperator{\colim}{colim}

\newcommand{\Mod}{\mathrm{Mod}}

\newcommand{\Fun}{\mathrm{Fun}}
\newcommand{\Ind}{\mathrm{Ind}}

\newcommand{\Nm}{\mathrm{Nm}}

\newcommand{\Gr}{\mathrm{Gr}}

\DeclareMathOperator{\Pic}{Pic}

\DeclareMathOperator{\Spec}{Spec}

\newcommand{\op}{\mathrm{op}}
\newcommand{\br}{\mathrm{br}}

\newcommand{\cHom}{\mathcal{H}\!{\it om}}
\DeclareMathOperator{\Hom}{Hom}

\DeclareMathOperator{\Ext}{Ext}
\DeclareMathOperator{\Aut}{Aut}

\DeclareMathOperator{\HH}{HH}
\DeclareMathOperator{\cHH}{\mathcal{HH}}
\DeclareMathOperator{\Map}{Map}

\newcommand{\re}{\mathrm{e}}
\newcommand{\ph}{\mathrm{ph}}

\DeclareMathOperator{\PGL}{PGL}

\newcommand{\svee}{\scriptscriptstyle\vee}
\newcommand{\id}{\mathrm{id}}
\newcommand{\pr}{\mathrm{pr}}

\newcommand{\ob}{\mathrm{ob}}

\newcommand{\Cat}{\mathrm{Cat}}
\newcommand{\Catcl}{\mathrm{Cat}^{\mathrm{cl}}}

\newcommand{\Ku}{\mathcal{K}u}

\DeclareMathOperator{\characteristic}{char}

\newcommand{\tH}{\widetilde{\rH}}

\newcommand{\rtop}{\mathrm{top}}
\newcommand{\disc}{\mathrm{disc}}

\newcommand{\an}{\mathrm{an}}

\newcommand{\Ktop}[1][]{\rK_{#1}^{\rtop}}

\newcommand{\CG}{\mathrm{CG}}

\def\blank{-}

\newcommand{\EPW}{\mathrm{EPW}}

\newcommand{\ch}{\mathrm{ch}}
\newcommand{\rk}{\mathrm{rk}}
\newcommand{\td}{\mathrm{td}}

\newcommand{\cO}{\mathcal{O}}
\newcommand{\cA}{\mathcal{A}}
\newcommand{\cC}{\mathscr{C}}
\newcommand{\cD}{\mathcal{D}}
\newcommand{\ccD}{\mathscr{D}}
\newcommand{\cE}{\mathcal{E}}
\newcommand{\cF}{\mathcal{F}}
\newcommand{\cG}{\mathcal{G}}
\newcommand{\cH}{\mathcal{H}}

\newcommand{\cK}{\mathcal{K}}
\newcommand{\cL}{\mathcal{L}}
\newcommand{\cM}{\mathcal{M}}
\newcommand{\cN}{\mathcal{N}}

\newcommand{\cS}{\mathcal{S}}

\newcommand{\cU}{\mathcal{U}}
\newcommand{\cV}{\mathcal{V}}
\newcommand{\cW}{\mathcal{W}}

\newcommand{\cY}{\mathcal{Y}}
\newcommand{\cZ}{\mathcal{Z}}

\newcommand{\rH}{\mathrm{H}}
\newcommand{\rh}{\mathrm{h}}
\newcommand{\rK}{\mathrm{K}}
\newcommand{\rS}{\mathrm{S}}
\newcommand{\rL}{\mathrm{L}}

\newcommand{\rR}{\mathrm{R}}

\newcommand{\rT}{\mathrm{T}}

\newcommand{\fp}{\mathfrak{p}}
\newcommand{\fm}{\mathfrak{m}}
\newcommand{\fU}{\mathfrak U}

\newcommand{\bC}{\mathbf{C}}

\newcommand{\bZ}{\mathbf{Z}}
\newcommand{\bP}{\mathbf{P}}
\newcommand{\bQ}{\mathbf{Q}}

\begin{document}

\title[Kuznetsov's conjecture via K3 categories and group actions]{Kuznetsov's Fano threefold conjecture via K3 categories and enhanced group actions}

\author{Arend Bayer}
\address{School of Mathematics and Maxwell Institute, University of Edinburgh, James Clerk Max\-well Building, Peter Guthrie Tait Road, Edinburgh, EH9 3FD,United Kingdom}
\email{arend.bayer@ed.ac.uk}

\author{Alexander Perry}
\address{Department of Mathematics, University of Michigan, Ann Arbor, MI 48109 \smallskip}
\email{arper@umich.edu}

\thanks{A.B. was supported by EPSRC grant EP/R034826/1, and the ERC Grant ERC-2018-CoG-819864-WallCrossAG. A.P. was partially supported by NSF grants DMS-2112747, DMS-2052750, and DMS-2143271, a Sloan Research Fellowship, and the Institute for Advanced Study.} 

%\date{\today}

\begin{abstract}
We settle the last open case of Kuznetsov's conjecture on the derived categories of Fano threefolds. 
Contrary to the original conjecture, we prove the Kuznetsov components of quartic double solids and Gushel--Mukai threefolds are never equivalent, 
as recently shown independently by Zhang. 
On the other hand, we prove the modified conjecture asserting their deformation equivalence. 
Our proof of nonequivalence combines a categorical Enriques-K3 
correspondence with the Hodge theory of categories. 
Along the way, we obtain a categorical description of the periods of Gushel--Mukai varieties, 
which we use to resolve a conjecture of Kuznetsov and the second author on the birational categorical Torelli problem, 
as well as to give a simple proof of a theorem of Debarre and Kuznetsov on the fibers of the period map.   
Our proof of deformation equivalence relies on results of independent interest about obstructions to enhancing group actions on categories. 
\end{abstract}

\maketitle

%%%%%%%%%%%%%%%%%%%%%%%%%%%%%%%%%%%%%%%%%%%%%%%%%%%%%%

\section{Introduction}
\label{section-intro} 

We work over the complex numbers. 
If $V$ is a Fano\footnote{Fano varieties are smooth by convention in this paper.} 
threefold of Picard number $1$ with ample generator 
$H \in \Pic(V)$, then the \emph{index} of $V$ is the integer $i$ such that 
$K_V = -iH$ and the \emph{degree} is the integer $d = H^3$. 
The classification of Fano threefolds  \cite{iskovskikh} 
shows that if $i = 4$ then $V \cong \bP^3$, 
if $i = 3$ then $V$ is a quadric, if $i = 2$ then $1 \leq d \leq 5$, 
and if $i = 1$ then $d$ is even, $d \neq 20$, and $2 \leq d \leq 22$. 
Moreover, for any pair $(i, d)$ satisfying these restrictions, 
there is a unique and explicitly described deformation class of Fano threefolds with these numerics. 
For instance:
\begin{itemize}
\item Fano threefolds $Y$ of Picard number $1$, index $2$, and degree $2$ are quartic double solids, i.e. double covers $Y \to \bP^3$ branched along a quartic surface. 
\item Fano threefolds $X$ of Picard number $1$, index $1$, and degree $10$ are Gushel--Mukai (GM) threefolds, i.e. either intersections 
$X = \Gr(2,5) \cap \bP^7 \cap Q$ of the Grassmannian $\Gr(2,5)$ with a codimension $2$ linear subspace and a quadric in the Pl\"{u}cker embedding (in which case $X$ is called \emph{ordinary}), or double covers $X \to \Gr(2,5) \cap \bP^6$ of a codimension $3$ linear section branched along a quadric section 
(in which case $X$ is called \emph{special}). 
\end{itemize} 

There are some curious classical ``coincidences'' between the families with numerical invariants $(i = 2, d)$ and $(i=1, 4d +2)$. 
For instance, the rationality of a (generic) Fano threefold in a given family is preserved under this correspondence (see \cite{beauville-luroth}). 
At the level of Hodge theory, the dimensions of the intermediate Jacobians on each side also match, except for $d = 1$ (see Remark~\ref{remark-degree-1} below). 

\subsection{Kuznetsov's conjecture} 
In \cite{kuznetsov-fano3fold} Kuznetsov suggested an intrinsic explanation for 
these coincidences, in terms of bounded derived categories of coherent sheaves. 
If $Y$ is a Fano threefold of Picard number $1$ and index~$2$, there is a semiorthogonal decomposition 
\begin{equation*}
\Db(Y) = \llangle \Ku(Y), \cO_Y, \cO_Y(H) \rrangle  
\end{equation*} 
where $\Ku(Y)$ is the subcategory --- now known as the \emph{Kuznetsov component} --- defined by 
\begin{equation}
\label{KuY}
\Ku(Y) = \set{ F \in \Db(Y) \st \Ext^\bullet(\cO_Y, F) = \Ext^\bullet(\cO_Y(H), F) = 0}. 
\end{equation} 
If $X$ is a Fano threefold of Picard number $1$, index $1$, and degree $d$, 
then $d = 2 g - 2$ for an integer $g \geq 2$ known as the \emph{genus} of $X$. 
If $g \geq 6$ is even, 
there is a semiorthogonal decomposition 
\begin{equation*}
\Db(X) = \llangle \Ku(X) , \cE, \cO_X \rrangle 
\end{equation*} 
where $\cE$ is a canonical exceptional rank $2$ vector bundle on $X$ constructed by Mukai (see \cite[Theorem~6.2]{BLMS}), 
and  
\begin{equation}
\label{KuX} 
\Ku(X) = \set{ F \in \Db(X) \st \Ext^\bullet(\cO_X, F) = \Ext^\bullet(\cE, F) = 0 }. 
\end{equation} 
For example, if $X$ is a GM threefold then $g = 6$ and $\cE$ is the pullback of the tautological rank $2$ subbundle on $\Gr(2,5)$. 

Kuznetsov conjectured the categories $\Ku(Y)$ for $Y$ of index $2$ and degree $d$ 
can be realized as $\Ku(X)$ for $X$ of index $1$ and degree $4d+2$. 
More precisely, 
let $\cM^i_d$ denote the moduli stack of Fano threefolds of Picard number $1$, index $i$, and degree $d$; 
this is a smooth irreducible stack of finite type (see \cite{modulifano3}). 

\begin{conjecture}[\cite{kuznetsov-fano3fold}]
\label{conjecture-fano3fold}
For $1 \leq d \leq 5$ 
there exists a correspondence $\cZ \subset \cM^2_d \times \cM^1_{4d+2}$ that is dominant over each factor and such that for any point $(Y, X) \in \cZ$ there is an equivalence of categories  
$\Ku(Y) \simeq \Ku(X)$. 
\end{conjecture}

\begin{remark}
In this paper, we work with enhanced triangulated categories (see \S\ref{conventions}), 
so by an equivalence 
$\Ku(Y) \simeq \Ku(X)$ we mean an equivalence of such enhanced categories; 
this amounts to $\Ku(Y) \simeq \Ku(X)$ being given by a Fourier--Mukai kernel on $Y \times X$. 
Technically, Kuznetsov's conjecture as stated in \cite{kuznetsov-fano3fold} only requires the existence of a triangulated equivalence $\Ku(Y) \simeq \Ku(X)$, but a different conjecture of Kuznetsov \cite[Conjecture 3.7]{kuznetsov-HPD} implies any such equivalence is of Fourier--Mukai type. 
In fact, in the cases of interest in this paper, the Fourier--Mukai type conjecture was recently proved in 
\cite{LPZ-FM}, so our assumption that all equivalences are enhanced is harmless. 
\end{remark} 

As evidence, Kuzntesov \cite{kuznetsov-fano3fold} proved Conjecture~\ref{conjecture-fano3fold} for $d =3,4,5$. 

\begin{remark}
\label{remark-degree-1} For $d = 1$ the conjecture fails, for the following reason.
If $V$ is a Fano threefold and $\cA \subset \Db(V)$ is any semiorthogonal component defined as the orthogonal to an exceptional sequence, then the HKR isomorphism and additivity of Hochschild 
homology gives an isomorphism $\HH_{1}(\cA) \cong \rH^{1,2}(V)$.  
Thus, for $(Y, X) \in \cM^2_d \times \cM^1_{4d+2}$, a necessary condition for the existence of an equivalence $\Ku(Y) \simeq \Ku(X)$ is that $h^{1,2}(Y) = h^{1,2}(X)$. 
This equality holds for $d = 2,3,4,5$, but it fails for $d = 1$, as then 
$h^{1,2}(Y) = 21$ while $h^{1,2}(X) = 20$. 
In fact, for $d = 1$ there is some subtlety in even defining $\Ku(X)$ --- note that in~\eqref{KuX} we excluded the case $g = 4$ --- but this argument applies 
to any possible definition of $\Ku(X)$. Instead, in an article in preparation Kuznetsov and Shinder show that $\Ku(X)$ and $\Ku(Y)$ are related by a degeneration and resolution: there exists a smooth proper family of categories with generic fiber $\Ku(X)$ and special fiber a categorical resolution of $\Ku(Y)$ for a nodal $Y$. 
\end{remark} 

Bernardara and Tabuada \cite{bernardara} observed that Conjecture~\ref{conjecture-fano3fold} 
also fails for $d = 2$, essentially for dimension reasons: 
the categories $\Ku(X)$ of GM threefolds vary in a $20$-dimensional family, while $\cM^2_2$ is only $19$-dimensional. 
Thus, if $\cZ \subset \cM^2_2 \times \cM^1_{10}$ is a correspondence parameterizing Fano threefolds with equivalent Kuznetsov components, then $\cZ$ does not dominate $\cM^1_{10}$. 
This left open the question of whether there could be such a correspondence dominating 
$\cM_2^2$, as suggested by the dimension count. 

\subsection{Main results} 
Our first main result says that, somewhat surprisingly, there does not even exist a nonempty 
correspondence $\cZ \subset \cM^2_2 \times \cM^1_{10}$ parameterizing equivalent 
Kuznetsov components, and thus Conjecture~\ref{conjecture-fano3fold} fails maximally for $d = 2$. 

\begin{theorem}
\label{main-theorem} 
Let $Y$ be a 
quartic double solid, and let $X$ be a GM threefold.
Then  $\Ku(Y)$ and $\Ku(X)$ are not equivalent.
\end{theorem}

\begin{remark}
This result was also recently shown by Zhang \cite{Zhang}, via a completely different method, using uniqueness of (Serre-invariant) Bridgeland stability conditions and moduli spaces of stable objects.
\end{remark} 

In view of the failure of Conjecture~\ref{conjecture-fano3fold} for $d = 2$, 
Kuznetsov suggested a weakening of the conjecture, which asserts that the categories $\Ku(Y)$ and $\Ku(X)$
are ``deformation equivalent''. 
Theorem~\ref{main-theorem} can be thought of as a negative result in this direction, as 
the simplest way the modified conjecture could be true is if $\Ku(Y) \simeq \Ku(X)$ for some 
$(Y, X) \in \cM^2_2 \times \cM^1_{10}$. 
Nonetheless, our second main result confirms Kuznetsov's modified conjecture. 

\begin{theorem}
\label{main-theorem-2}
There exists a smooth pointed curve $(B, o)$ and a smooth proper $B$-linear category $\cC$ such that: \begin{enumerate}
\item The fiber $\cC_o$ is equivalent to $\Ku(Y)$ for a quartic double solid $Y$. 
\item For $b \in B \setminus \{ o \}$, the fiber $\cC_b$ is equivalent to $\Ku(X_b)$ for a 
GM threefold $X_b$. 
\end{enumerate} 
\end{theorem}

We expect this result 
%Theorem~\ref{main-theorem-2} 
to be a useful tool for relating Bridgeland moduli spaces of objects in Kuznetsov components of quartic double solids and GM threefolds. 

\begin{remark}
See \S\ref{section-linear-categories} for a summary on $B$-linear categories. In general, a $B$-linear category should be thought 
of as a ``family of categories parameterized by $B$''; there is a well-behaved notion of base change for such 
categories, which in particular gives rise to a $\kappa(b)$-linear fiber category $\cC_b$ for any $b \in B$. 
Theorem~\ref{main-theorem-2} thus informally says that the Kuznetsov components of 
GM threefolds smoothly specialize to those of quartic double solids. 
\end{remark}

Our method of proof of Theorem~\ref{main-theorem} naturally leads to our third main result, concerning the ``categorical Torelli problem''. 
Namely, the intermediate Jacobian of a Fano threefold $X$ is determined by its Kuznetsov component  \cite{IHC-CY2}, and hence the association $X \mapsto \Ku(X)$ can be thought of as a categorical lift of the period map. 
The categorical Torelli problem then asks to what extent $X$ is determined by $\Ku(X)$. 
Positive answers are known in many situations; see \cite{categorical-torelli-survey} for a recent survey and references. 
One particularly interesting open case is that of 
GM threefolds, for which the $3$-dimensional case of a conjecture of Kuznetsov and the second author predicts the following. 

\begin{conjecture}[{\cite[Conjecture 1.7]{categorical-cones}}] 
\label{conjecture-birational-GM}
If $X_1$ and $X_2$ are GM threefolds such that there is an equivalence $\Ku(X_1) \simeq \Ku(X_2)$, 
then $X_1$ and $X_2$ are birational. 
\end{conjecture} 

We note that by the duality conjecture \cite[Conjecture 3.7]{GM-derived} proved in \cite{categorical-cones}, there are indeed $2$-dimensional families of birational GM threefolds with equivalent Kuznetsov component. 
More precisely, \cite{DebKuz:birGM} introduces a notion of \emph{period partnership} and \emph{duality} for GM varieties (see \S\ref{subsection-application-periods} for definitions), shows that these relations imply birationality, and explicitly describes the locus of period partners and duals of a given GM variety in terms of an associated EPW sextic (the locus being $2$-dimensional for a GM threefold), while the duality conjecture implies that the Kuznetsov components of period partners or duals are equivalent. 
We show that this is in fact the \emph{only} way for GM threefolds to have equivalent Kuznetsov components, and therefore resolve Conjecture~\ref{conjecture-birational-GM} while simultaneously computing the fiber of the ``categorical period map'': 

\begin{theorem}
\label{main-theorem-3} 
Let $X_1$ and $X_2$ be GM threefolds. 
Then $\Ku(X_1) \simeq \Ku(X_2)$ if and only if 
$X_1$ and $X_2$ are period partners or duals. 
In particular, if $\Ku(X_1) \simeq \Ku(X_2)$, then $X_1$ and $X_2$ are birational. 
\end{theorem} 

One of the appeals of Theorem~\ref{main-theorem-3} is that the expected corresponding result for the ordinary period map, i.e. with Kuznetsov components replaced by intermediate Jacobians, is currently unknown. 
This illustrates the utility of the extra structure provided by working categorically. 

\begin{remark}
Under a genericity assumption, the birationality of $X_1$ and $X_2$ above was recently shown in \cite{categorical-torelli-GM3}, via a completely different method, and without the genericity assumption in the upcoming \cite{JLZ:Brill-Noether}.
\end{remark} 

The proofs of Theorems~\ref{main-theorem},~\ref{main-theorem-2}, and~\ref{main-theorem-3} involve ideas of independent interest, sketched below. 

\subsubsection*{Categorical Enriques-K3 correspondence} The Kuznetsov components of quartic double solids and GM threefolds are \emph{Enriques categories}, in the sense that their Serre functors are of the form $\tau \circ [2]$ where $\tau$ is a nontrivial involution generating a $\bZ/2$-action. 
To any Enriques category, there is an associated \emph{2-Calabi--Yau (CY2) cover}, defined as the invariant category for the $\bZ/2$-action. 
By \cite{cyclic-covers}, if $Y \to \bP^3$ is a quartic double solid with branch locus a quartic K3 surface $Y_{\br} \subset \bP^3$, 
then the CY2 cover of $\Ku(Y)$ is $\Db(Y_{\br})$, while if $X$ is a GM threefold, the CY2 cover of $\Ku(X)$ is the Kuznetsov component of the ``opposite'' GM variety $X^{\op}$ 
(a Fano fourfold if $X$ is ordinary or a K3 surface if $X$ is special, see Definition~\ref{definition-opposite}); these CY2 categories are called \emph{K3 categories} because their Hochschild homology agrees with that of a K3 surface. 

As reviewed in \S\ref{section-enriques}, the CY2 cover of an Enriques category admits a \emph{residual $\bZ/2$-action}, which should be thought of as an analogue of the covering involution of a K3 surface over an Enriques surface; in the case of $\Db(Y_{\br})$ and $\Ku(X^{\op})$ these actions can be described explicitly, see Theorem~\ref{theorem-K3-cover-examples}. 
One of our key observations is that two Enriques categories are equivalent if and only if 
their CY2 covers are $\bZ/2$-equivariantly equivalent (Lemma~\ref{lemma-K3-cover-equivalence}). 
This is useful as K3 categories are often easier to understand. 

\subsubsection*{Outline of the proof of Theorem~\ref{main-theorem}} In particular, 
Theorem~\ref{main-theorem} reduces to proving the nonexistence of a $\bZ/2$-equivariant equivalence 
$\Db(Y_{\br}) \simeq \Ku(X^{\op})$. 
To rule out such an equivalence, we study the induced $\bZ/2$-equivariant isometry $\tH(Y_{\br}, \bZ) \cong \tH(\Ku(X^{\op}), \bZ)$ between their Mukai Hodge structures, whose definitions are reviewed in \S \ref{mukai-HS}. This leads to a contradiction
to constraints on the 
periods of GM fourfolds when $X$ is ordinary (Theorem~\ref{theorem-image-period-morphism}), and
those of GM surfaces when $X$ is special (Lemma~\ref{lem:GMK3periods}). 

\subsubsection*{Enhanced group actions} 
The above discussion elided a subtlety about $\bZ/2$-actions. In general, if $G$ is a finite group, then there are several possible notions of an action of $G$ on a category $\cC$. Naively, one might consider a homomorphism $\phi$ from $G$ to the group of autoequivalences modulo isomorphisms of functors. 
However, more structure is needed to define a reasonable category $\cC^G$ of $G$-equivariant objects in $\cC$; namely, following Deligne \cite{Deligne:groupedestresses}, we need to specify suitably compatible isomorphisms of functors $\phi(g) \circ \phi(g') \cong \phi(g \cdot g)$.  This suffices if $\cC$ is an ordinary category, but if $\cC$ is triangulated then in general $\cC^G$ need not be (see \cite[Theorem 6.9]{elagin} for a sufficient condition). 

To correct this, we instead work with an enhanced triangulated category $\cC$ --- we use $\infty$-categorical enhancements, see \S\ref{conventions} --- and consider $\infty$-categorical group actions on $\cC$. 
Then there is a well-behaved category $\cC^G$ of invariants, but the price we pay is that it is a priori much harder to specify a group action of $G$ on $\cC$, as it requires an infinite hierarchy of data. At the first two levels, if $\rh \cC$ denotes the triangulated homotopy category of $\cC$, then an $\infty$-categorical action on $\cC$ determines both a naive $G$-action on $\rh \cC$ given by a homomorphism $\phi$ as above, as well as a $G$-action on $\rh \cC$ in the sense of Deligne; we call the former a \emph{$1$-categorical action} on $\cC$, and the latter a \emph{$2$-categorical action}. 
We study obstructions to and uniqueness of lifts of $1$- and $2$-categorical actions to $\infty$-categorical ones. 
In particular, we show that if $\cC$ satisfies a connectivity hypothesis on Hochschild cohomology (which holds for most  categories of interest), then given a $1$-categorical action there is a single obstruction to the existence of an $\infty$-categorical lift, and the set of lifts form a torsor over an explicit cohomology group (Corollary~\ref{corollary-obstruction-HH}); 
moreover, this obstruction and torsor are the exact same as those controlling $2$-categorical lifts. 
This in particular answers a question raised in \cite{elagin}. 
In the special case where $\cC$ is the derived category of a variety, 
the result is the following enhancement of \cite[Theorem~2.1]{beckmann-oberdieck}. 

\begin{theorem}
Let $X$ be a connected smooth proper variety over a field $k$. 
Let $G$ be a finite group with a group homomorphism $\phi$ to the group of autoequivalences of $\Db(X)$.
Then there is a canonical obstruction class $\ob(\phi) \in \rH^3(BG, k^{\times})$, where the $G$-action on $k^{\times}$ is trivial, such that an $\infty$-categorical lift of $\phi$ exists if and only if $\ob(\phi) = 0$, in which case the set of equivalence classes of such lifts is an $\rH^2(BG, k^{\times})$-torsor. 
\end{theorem}

The complete result in Corollary~\ref{corollary-obstruction-HH} applies 
in the relative setting where we consider categories $\cC$ that are linear over a base scheme, instead of merely a field. 
An important technical result is that the vanishing of the obstruction mentioned above is an open condition in the \'{e}tale topology of the base (Proposition~\ref{proposition-etale-vanishing}). 

\subsubsection*{Outline of the proof of Theorem~\ref{main-theorem-2}} 
By our discussion above, if $X$ is a special GM threefold and $Y$ is a quartic double solid, then the CY2 covers of their Kuznetsov components are $\Db(X^{\op})$ and $\Db(Y_{\br})$, where $X^{\op}$ is a GM K3 surface and $Y_{\br}$ is a quartic K3 surface, and the Kuznetsov components can be recovered as the invariant categories for the residual $\bZ/2$-actions. 
The idea of the proof of Theorem~\ref{main-theorem-2} is thus to find a specialization of $X^{\op}$ to a quartic K3, with a $\bZ/2$-action that restricts on fibers to the residual $\bZ/2$-actions. 
To do so, we first construct such a specialization with a $1$-categorical action on the family of derived categories of the K3 surfaces, and then use the general results discussed above to lift this to an $\infty$-categorical action. 
Passing to invariant categories gives the category $\cC$ promised by Theorem~\ref{main-theorem-2}, which is smooth and proper by a general result (Proposition~\ref{proposition-invariants-smooth-proper}) that we prove. 

\subsubsection*{The role of derived algebraic geometry} We briefly explain the role played by derived algebraic geometry, more specifically stable $\infty$-categories, in this paper. We need the notion of a category linear over a base $B$, along with base change. Sometimes, one considers such categories as admissible $B$-linear subcategories of $\Dperf(X)$ for a scheme $X$ over $B$; this is e.g.~the approach taken in \cite{stability-families}. However, we know of no such embedding of the category $\cC$ in Theorem~\ref{main-theorem-2}.

Instead, as explained above $\cC$ is constructed as the invariant category for a $\bZ/2$-action on the derived category $\Dperf(\cS)$ of the total space $\cS$  of a family of K3 surfaces over $B$. 
A result of Elagin \cite[Corollary 6.10]{elagin} allows to construct from a $2$-categorical action on $\Dperf(\cS)$ a triangulated structure on the invariant category, but it does not come with a natural $B$-linear structure that satisfies base change. 
Instead, our results in Section~\ref{section-group-actions-cats} show that the $\bZ/2$-action on the triangulated $\Dperf(\cS)$ lifts to an action on its enhancement as a $\infty$-category; then the desired properties of the invariant category are automatic.

\subsubsection*{Categorical description of periods} 
As a byproduct of our proof of Theorem~\ref{main-theorem}, we obtain 
a categorical description of the periods of even-dimensional GM varieties. 
GM varieties are generalizations of GM threefolds to dimensions $2 \leq n \leq 6$, with similarly defined Kuznetsov components (see Definition~\ref{definition-GM} and~\eqref{DbW}). 
If $W$ is such a variety of dimension $4$ or $6$, then $\Ku(W)$ is canonically the CY2 cover of an Enriques category (generalizing the discussion above for $W = X^{\op}$), and hence carries a canonical residual $\bZ/2$-action. 
If $n = \dim(W)$, the period map assigns to $W$ the Hodge structure $\rH^n(W, \bZ)_0$ given as the 
orthogonal to the sublattice $\rH^n(\Gr(2,5), \bZ) \subset \rH^n(W, \bZ)$. 

\begin{proposition}
\label{proposition-GM-periods}
Let $W$ be a GM variety of dimension $n = 4$ or $6$. 
Let $\tH(\Ku(W), \bZ)_0$ denote
the orthogonal to 
the invariant sublattice $\tH(\Ku(W), \bZ)^{\bZ/2} \subset \tH(\Ku(W), \bZ)$ for the residual $\bZ/2$-action. 
Then there is an isometry of weight $2$ Hodge structures 
\begin{equation*}
\tH(\Ku(W), \bZ)_0 \cong \rH^n(W, \bZ)_0(\tfrac{n}{2}-1), 
\end{equation*} 
where $(\frac{n}{2}-1)$ on the right denotes a Tate twist. 
\end{proposition} 

Our main application, to Theorem~\ref{main-theorem-3}, is explained below. 
As another application, we give a simple proof of a recent result of Debarre and Kuznetsov \cite{DebKuz:periodGM}, which identifies the periods of even-dimensional GM varieties that are ``generalized partners or duals'' (Theorem~\ref{theorem-periods-DK}), and implies the period map factors through the moduli space of double EPW sextics. 

\begin{remark}
\label{remark-IJ}
In \cite{IHC-CY2}, canonical weight $0$ and $-1$ Hodge structures $\Ktop[0](\cA)$ and $\Ktop[1](\cA)$ are constructed for any admissible subcategory $\cA$ of the derived category of a smooth proper variety, which can be thought of as versions of even and odd degree cohomology of $\cA$. 
If $\cA = \Ku(W)$ for an even-dimensional GM variety, then $\Ktop[0](\cA)$ is up to Tate twist the Mukai Hodge structure. 
For many odd-dimensional Fano varieties (including GM varieties), if $\cA$ is taken to be 
an appropriate Kuznetsov component, then $\Ktop[1](\cA)$ recovers the middle Hodge structure of $W$ on the nose. 
In this way, the categorical description of periods in even dimensions is more 
subtle than in odd dimensions. 
\end{remark}

\subsubsection*{Outline of the proof of Theorem~\ref{main-theorem-3}} 
If $X_1$ and $X_2$ are GM threefolds with $\Ku(X_1) \simeq \Ku(X_2)$, 
then passing to CY2 covers we obtain an equivalence $\Ku(X_1^{\op}) \simeq \Ku(X_2^{\op})$ equivariant for the residual $\bZ/2$-actions. 
By a trick one can reduce to the case where $X_1$ and $X_2$ are ordinary, so that $X_1^{\op}$ and $X_2^{\op}$ are GM fourfolds. 
Then Proposition~\ref{proposition-GM-periods} implies an isometry of Hodge structures $\rH^4(X_1^{\op}, \bZ)_0 \cong \rH^4(X_2^{\op}, \bZ)_0$. 
Theorem~\ref{main-theorem-3} then follows by combining this with the factorization of the period map through the moduli space of double EPW sextics (mentioned above) and the injectivity of the period map for double EPW sextics (by Verbitsky's Torelli theorem). 
The moral of this argument is that passing to CY2 covers allows us to leverage Torelli theorems for hyperk\"{a}hler fourfolds. 

\subsection{Further conjectures and questions} 
We highlight several further directions 
suggested by our work. 

\subsubsection*{Birational geometry and intermediate Jacobians of the threefolds} 
Heuristic relations between derived categories and birational geometry \cite{kuznetsov-rationality} suggest that if $Y$ were a quartic double solid which is birational to a GM threefold $X$, then $\Ku(Y) \simeq \Ku(X)$. 
Together with Theorem~\ref{main-theorem}, this leads to: 

\begin{conjecture}
\label{conjecture-YX-not-birational}
Let $Y$ be a quartic double solid, and let $X$ be a GM threefold.
Then $Y$ is not birational to $X$. 
\end{conjecture} 

Motivated by this, we also propose: 

\begin{conjecture}
\label{conjecture-JY-not-JX}
Let $Y$ be a quartic double solid, and let $X$ be a GM threefold.
Then the intermediate Jacobian $J(Y)$ is not isomorphic to $J(X)$ as a principally polarized abelian variety. 
\end{conjecture} 

We note that $J(X)$ and $J(Y)$ are both $10$-dimensional. 
One could hope to address Conjecture~\ref{conjecture-JY-not-JX} by 
proving a description for the singular locus of the theta divisor of $J(X)$ in terms of Bridgeland moduli spaces for $\Ku(X)$. 
Our main interest in 
Conjecture~\ref{conjecture-JY-not-JX} is that it explains both Theorem~\ref{main-theorem} and Conjecture~\ref{conjecture-YX-not-birational}.  

\begin{proof}[Conjecture~\ref{conjecture-JY-not-JX} $\implies$ Theorem~\ref{main-theorem}]
By \cite[Lemma 5.30]{IHC-CY2} (cf.~Remark~\ref{remark-IJ}) an equivalence $\Ku(Y) \simeq \Ku(X)$ would imply an isomorphism $J(Y) \cong J(X)$ of principally polarized abelian varieties. 
\end{proof} 

\begin{proof}[Conjecture~\ref{conjecture-JY-not-JX} $\implies$ Conjecture~\ref{conjecture-YX-not-birational}]
For a Fano threefold $W$, the Clemens-Griffiths component $J_{\CG}(W)$ --- defined as the product of the principally polarized factors of $J(W)$ that are 
not Jacobians of curves --- is a birational invariant \cite{clemens-griffiths}. 
It follows from \cite{quartic-double-irrational} that $J_{\CG}(Y) = J(Y)$, so if $Y$ is birational to $X$ then we must have $J(Y) \cong J(X)$. 
\end{proof} 

\begin{remark}
Conjecture~\ref{conjecture-JY-not-JX}, and hence Conjecture \ref{conjecture-YX-not-birational}, are shown for a \emph{generic} GM threefold $X$ in \cite[Corollary 7.6]{DIM3fold}. 
\end{remark}

\subsubsection*{Loci of equivalent Kuznetsov components}
Conjecture~\ref{conjecture-fano3fold} motivates studying in general the 
locus where Kuznetsov components are equivalent in families of Fano varieties. 
We note the following consequence of work of Anel and To\"{e}n {\cite[Corollaire 3.3]{anel-toen}}. 

\begin{theorem}
\label{theorem-Z}
For $j = 1,2$, let $\cC_j$ be a smooth, proper, connected $S_j$-linear category for a scheme $S_j$. 
Then the locus $\cZ \subset S_1 \times S_2$ of points $(s_1, s_2) \in S_1 \times S_2$ such that $(\cC_1)_{s_1} \simeq (\cC_2)_{s_2}$ is a countable union of locally closed subspaces. 
\end{theorem} 

In \S\ref{section-HH-lincats}, we review the notion of connectedness for a linear category, 
and explain a useful criterion (Corollary~\ref{corollary-connected}) for checking it based on \cite{kuznetsov-heights}. 
The criterion applies to the Kuznetsov components of many Fano varieties, including those of Fano threefolds of Picard number $1$\footnote{See \cite[\S6]{BLMS} for a (somewhat ad hoc) definition of the Kuznetsov component without restrictions on the degree.} and of cubic fourfolds \cite{kuznetsov-cubic}. 
By Theorem~\ref{theorem-Z} the locus where connected Kuznetsov components become equivalent in the product of two moduli spaces is a countable union of locally closed subspaces. 
We believe such loci deserve further study. For instance: 

\begin{question}
\label{question-equivalences-specialize}
When do equivalences between Kuznetsov components specialize, i.e. when are the locally closed 
subspaces above actually closed? 
\end{question}

This specialization property holds in the setting of Conjecture~\ref{conjecture-fano3fold}: for $d = 3,4,5$ by the results of \cite{kuznetsov-fano3fold} combined with the categorical Torelli theorems for 
Fano threefolds $Y$ of Picard number $1$, index $2$, and degree $d=3,4,5$; 
for $d = 2$ by Theorem~\ref{main-theorem}; and for $d = 1$ by Remark~\ref{remark-degree-1}. 
For cubic fourfolds, the results of \cite{stability-families} give a partial answer to Question~\ref{question-equivalences-specialize}: they imply that derived equivalences between Kuznetsov components of cubic fourfolds and K3 surfaces specialize, as they are given by $2$-dimensional moduli spaces of stable objects.

\subsection{Related work} 
There are alternative approaches to some of our results: 
Theorem~\ref{main-theorem} was proved independently by Zhang in \cite{Zhang}, based on a study of Bridgeland moduli spaces, while Theorem~\ref{main-theorem-2} was proved independently by Kuznetsov and Shinder in work in preparation \cite{KS-fano3fold} (see also \cite[\S5.4]{kuznetsov-sod-families}), based on a degeneration argument and a theory of ``absorption of singularities''. 
Our paper and these two use completely different methods, which we believe are interesting in their own right. 

After posting the first version of this paper, we learned that Zo\"{e} Schroot has obtained results similar to ours in \S\ref{section-group-actions-cats} on enhancing group actions on categories. 

We also note that some of the ideas in this paper are used in the upcoming \cite{moduli-enriques} to describe Bridgeland moduli spaces for Enriques categories, like the Kuznetsov components of GM threefolds or quartic double solids. 

\subsection{Organization of the paper}
In \S\ref{section-HH-lincats} we review the formalism of enhanced triangulated categories linear over a base scheme, as well as their Hochschild cohomology. 
In \S\ref{section-group-actions-cats} we discuss $\infty$-categorical group actions on (linear) categories, and in particular study the obstruction to lifting an action on the homotopy category to the $\infty$-level. 
In \S\ref{section-enriques} we explain the correspondence between Enriques categories and their CY2 covers. 
In \S\ref{section-Hodge-theory} we review Mukai Hodge structures and prove Proposition~\ref{proposition-GM-periods}. 
In \S\ref{section-torelli} we prove Theorem~\ref{main-theorem-3},  
in \S\ref{section-nonexistence} we prove Theorem~\ref{main-theorem}, and finally 
in \S\ref{section-deformation-equivalence} we prove Theorem~\ref{main-theorem-2}.

\subsection{Conventions} 
\label{conventions} 

Schemes are tacitly assumed to be quasi-compact and quasi-separated. 
A variety over a field $k$ is an integral scheme which is separated and of finite type over $k$. Fano varieties are smooth by convention. 

For a scheme $X$, $\Dperf(X)$ denotes the category of perfect complexes, $\Dqc(X)$ denotes the unbounded derived category of quasi-coherent sheaves, and $\Db(X)$ denotes the bounded derived category of coherent sheaves. (In fact in all cases where we consider $\Db(X)$ in this paper, $\Db(X) = \Dperf(X)$, so it is just a matter of notation.)
All functors are derived by convention. 
In particular, for a morphism $f \colon X \to Y$ of schemes we write $f_*$ and $f^*$ for the derived pushforward and pullback functors, and for $E, F \in \Dperf(X)$ we write $E \otimes F$ for the derived tensor product. 

For technical convenience all categories in the paper are considered as enhanced categories. 
More precisely, instead of $k$-linear triangulated categories we consider $k$-linear categories and functors between them in the sense of \cite{NCHPD}, i.e. we consider small idempotent-complete stable $\infty$-categories equipped with a module structure over $\Db(\Spec(k))$ (or equivalently, when $\characteristic(k) = 0$, DG categories over $k$). 
In particular, for a variety $X$ over $k$ we regard 
$\Db(X)$ as such a category; its homotopy category $\mathrm{h}\Db(X)$ is the classical triangulated derived category. 
We note that giving a semiorthogonal decomposition of $\Db(X)$ is equivalent to giving one of $\mathrm{h}\Db(X)$, and 
by the results of \cite{bzfn} and \cite{kuznetsov-bc} if 
$\cC \subset \Db(X)$ and $\ccD \subset \Db(Y)$ are semiorthogonal components of smooth proper varieties, then any functor $\cC \to \ccD$ of $k$-linear categories is induced by a Fourier--Mukai kernel on $X \times Y$. 

If $\cC$ is a $k$-linear category, we use the notation $\cHom_k(E,F) \in \Dqc(\Spec(k))$ for the mapping object between objects $E, F \in \cC$, see \S\ref{section-linear-categories}; 
in case $\cC \subset \Db(X)$ is a semiorthogonal component of the derived category of a variety, 
then $\cHom_k(E, F)$ coincides with the classical derived Hom complex $\mathrm{RHom}(E, F)$.

In several places, we also need the general notion of categories linear over a base scheme, briefly reviewed in \S\ref{section-linear-categories}. 

As stated in the introduction, our main results are over the complex numbers, and correspondingly in \S\ref{section-Hodge-theory}-\S\ref{section-deformation-equivalence} we work in this setting. 
However, in the foundational part of the paper,   \S\ref{section-HH-lincats}-\S\ref{section-enriques}, 
we work over more general bases, as explained there. 

\subsection{Acknowledgements} 
We thank Sasha Kuznetsov, Laura Pertusi, and Xiaolei Zhao for 
helpful discussions about this work. 
We are especially grateful to Bhargav Bhatt for explaining to us the proof of Lemma~\ref{lemma-etale-vanishing-ring} and suggesting Example~\ref{example-etale-vanishing}, and to Sasha Kuznetsov for carefully reading a preliminary version of this paper. We also thank the referee for their careful reading.  

%%%%%%%%%%%%%%%%%%%%%%%%%%%%%%%%%%%%%%%%%%%%%%%%%%%%%%

\section{Hochschild cohomology of linear categories} 
\label{section-HH-lincats}

In this section we discuss categories linear over a base scheme 
and their Hochschild cohomology. 
In \S\ref{section-linear-categories} we recall some of the basic formalism of linear categories, 
in \S\ref{section-HH} we define Hochschild cohomology and review some of its properties, and 
in \S\ref{section-connected} we define the notion of connectedness of a linear category (which appears as a hypothesis in Theorem~\ref{theorem-Z}) 
and explain a convenient method for checking it in practice. 

\subsection{Linear categories} 
\label{section-linear-categories} 
Fix a base scheme $S$. 
We use the notion of $S$-linear categories as in \cite{NCHPD}. 
Namely, 
the derived category $\Dperf(S)$ is a commutative algebra object in the 
$\infty$-category of small idempotent-complete stable $\infty$-categories, and an 
\emph{$S$-linear category} is a module object over $\Dperf(S)$; 
in particular, an $S$-linear category $\cC$ is equipped with an 
action functor $\Dperf(S) \times \cC \to \cC$. 

There is a well-behaved base change operation along any morphism $T \to S$ which produces a $T$-linear category 
\begin{equation*} 
\cC_T = \cC \otimes_{\Dperf(S)} \Dperf(T). 
\end{equation*}
This construction is compatible with semiorthogonal decompositions in the following sense. 
We say a semiorthogonal decomposition $\cC = \llangle \cC_1, \dots, \cC_n \rrangle$ is 
\emph{$S$-linear} if the $\Dperf(S)$-action preserves each of the components $\cC_i$, in which case $\cC_i$ inherits the structure of an $S$-linear category. Then for any morphism 
$T \to S$, there is an induced $T$-linear semiorthogonal decomposition 
\begin{equation*} 
\cC_T = \llangle (\cC_1)_T, \dots, (\cC_m)_T \rrangle . 
\end{equation*} 

For a point $s \in S$, we use the following terminology. The \emph{fiber} $\cC_s$ of $\cC$ over $s$ is the base change along $\Spec(\kappa(s)) \to S$. Similarly, if $F \colon \cC \to \cD$ is a functor between $S$-linear categories, its \emph{fiber} over $s$ is the functor $F_s \colon \cC_s \to \ccD_s$ obtained by base change. 

We also recall that for objects $E, F \in \cC$, there is a mapping object $\cHom_S(E,F) \in \Dqc(S)$ characterized by equivalences 
\begin{equation} 
\label{mapping-object}
\Map_{\Dqc(S)}(G, \cHom_S(E,F)) \simeq \Map_{\cC}(E \otimes G , F) 
\end{equation} 
where $\Map(-,-)$ denotes the space of maps in an $\infty$-category. 

\begin{example}
Let $f \colon X \to S$ be a morphism of schemes. 
Then $\Dperf(X)$ is naturally an $S$-linear category. 
By \cite{bzfn}, for $T \to S$ the base changed category $\Dperf(X)_T$ recovers $\Dperf(X_T)$, where $X_T$ is the derived base change (which agrees with the classical base change if e.g.~$X \to S$ is flat) of $X$ along $T \to S$.
Further, if $E, F \in \Dperf(X)$ then 
\begin{equation*}
\cHom_S(E, F) \simeq f_* \cHom_X(E,F) 
\end{equation*} 
where $\cHom_X(E,F) \in \Dqc(X)$ denotes the derived sheaf Hom on $X$. 
\end{example}

\subsection{Hochschild cohomology}
\label{section-HH}

Recall that given two $S$-linear categories $\cC$ and $\ccD$, 
there is a natural $S$-linear category $\Fun_S(\cC, \ccD)$ whose objects are the $S$-linear functors $\cC \to \ccD$. 

\begin{definition}
Let $\cC$ be an $S$-linear category. 
The \emph{sheafy Hochschild cohomology over $S$} of $\cC$ is 
\begin{equation}
\label{HH-definition}
\cHH^*(\cC/S) \coloneqq \cHom_S(\id_{\cC}, \id_{\cC}) \in \Dqc(S), 
\end{equation} 
i.e. the endomorphism object of $\id_{\cC}$ regarded as an object of $\Fun_S(\cC, \cC)$. 
The \emph{Hochschild cohomology over $S$} of $\cC$ is the derived global sections 
\begin{equation*}
\HH^*(\cC/S) \coloneqq \rR\Gamma(\cHH^*(\cC/S)) \in \mathrm{D}(\Mod_{\Gamma(S, \cO_S)})
\end{equation*} 
For $i \in \bZ$ we denote by $\cHH^i(\cC/S)$ the $i$-th cohomology sheaf of $\cHH^*(\cC/S)$, 
and by $\HH^i(\cC/S)$ the $i$-th cohomology module of $\HH^*(\cC/S)$. 
When $\cC = \Dperf(X)$ for a morphism of schemes $X \to S$, we use the simplified notation 
$\cHH^*(X/S)$ and $\HH^*(X/S)$ for Hochschild cohomology. 
\end{definition} 

\begin{warning}
Sometimes different notation is used for Hochschild cohomology; for instance, in 
\cite{IHC-CY2} which we shall reference several times below, 
sheafy Hochschild cohomology is denoted by $\HH^*(\cC/S)$. 
\end{warning}  

\begin{example}
\label{example-HHX}
Let $X \to S$ be a morphism of schemes. 
Then there is an equivalence 
\begin{equation*}
\cHH^*(X/S) \simeq \cHom_S(\cO_{\Delta}, \cO_{\Delta})
\end{equation*} 
where $\cO_{\Delta} \in \Dqc(X \times_S X)$ is the structure sheaf of the diagonal 
$\Delta \subset X \times_S X$. 
Indeed, $\cHH^*(X/S)$ can be computed as the Hochschild cohomology of the presentable $S$-linear category $\Dqc(X) = \Ind(\Dperf(X))$ (see \cite[Remark 4.2]{IHC-CY2}), and then the claim follows from the equivalence $\Dqc(X \times_S X) \simeq \Fun_{S}(\Dqc(X), \Dqc(X))$ of \cite{bzfn} which sends $\cO_{\Delta}$ to $\id_{\Dqc(X)}$. 
\end{example}

\begin{remark}
\label{remark-HH-bc} 
Sheafy Hochschild cohomology satisfies a base change formula: 
if $\cC$ is an $S$-linear category and $g \colon T \to S$ is a morphism, then there is a natural equivalence  
\begin{equation}
\label{cHH-bc}
g^*\cHH^*(\cC/S)  \xrightarrow{\, \sim \,} \cHH^*(\cC_T/T). 
\end{equation} 
This morphism is constructed as follows. The base change formalism gives a functor 
\begin{equation*}
\Fun_S(\cC, \cC) \otimes_{\Dperf(S)} \Dperf(T) \to \Fun_T(\cC_T, \cC_T), 
\end{equation*} 
which induces a morphism on mapping objects 
\begin{equation*}
g^*\cHH^*(\cC/S) \simeq \cHom_T(\id_\cC \boxtimes \cO_T, \id_{\cC} \boxtimes \cO_T) \to \cHom_T(\id_{\cC_T}, \id_{\cC_T}) = \cHH^*(\cC_T/T) , 
\end{equation*} 
where the first equivalence is the K\"{u}nneth formula for mapping objects \cite[Lemma 2.10]{NCHPD}. 
That this morphism is an equivalence is proved (in a more general context) in \cite[Lemma 4.3]{IHC-CY2}.  

Base change also induces maps on Hochschild cohomology. Namely, taking $\rR\Gamma(-)$ of the adjoint of the map~\eqref{cHH-bc} gives a natural map  
\begin{equation*}
\HH^*(\cC/S)  \to \HH^*(\cC_T/T) 
\end{equation*}
(which is usually not an equivalence). 
\end{remark} 

We will be concerned with the case of smooth proper $S$-linear categories. 
We refer to \cite[\S4]{NCHPD} for background on this notion. 
In particular, we note that if $\cC$ is an $S$-linear semiorthogonal component of $\Dperf(X)$ where $X \to S$ is a smooth proper morphism, then $\cC$ is a smooth proper $S$-linear category \cite[Lemma 4.9]{NCHPD}. 
Although not strictly necessary for our purposes, 
we observe that in the smooth proper case Hochschild cohomology satisfies the following finiteness property. 

\begin{lemma}
Let $\cC$ be a smooth proper $S$-linear category. 
Then $\cHH^*(\cC/S) \in \Dperf(S)$ is a perfect complex. 
\end{lemma}

\begin{proof}
By definition, it suffices to show the functor category $\Fun_S(\cC, \cC)$ is proper. 
In fact, it is smooth and proper. 
Indeed, an $S$-linear category is smooth and proper if and only if it is dualizable, in which case the dual 
is given by the opposite category $\cC^{\op}$ (\cite[Lemma 4.8]{NCHPD}). 
Thus there is an equivalence $\Fun_S(\cC, \cC) \simeq \cC^{\op} \otimes_{\Dperf(S)} \cC$, and it follows that $\Fun_S(\cC, \cC)$ is dualizable, being the tensor product of such categories. 
\end{proof} 

\subsection{Connected linear categories} 
\label{section-connected}
Note that for any $S$-linear category $\cC$, by the definition of $\cHH^*(\cC/S)$ 
there is a canonical morphism $\cO_S \to \cHH^0(\cC/S)$. 

\begin{definition}
\label{definition-connected}
Let $\cC$ be an $S$-linear category. We say $\cC$ is \emph{connected (over $S$)} if for 
every morphism of schemes $T \to S$, we have 
$\cHH^i(\cC_T/T) = 0$ for $i < 0$ and $\cO_T \to \cHH^0(\cC_T/T)$ is an isomorphism. 
\end{definition} 

\begin{remark}
By base change for sheafy Hochschild cohomology (Remark~\ref{remark-HH-bc}), 
it suffices to consider only affine schemes $T$ in the definition of connectedness. 
\end{remark} 

The source of this terminology is the following example. 
\begin{example}
\label{example-scheme-connected}
Let $f \colon X \to S$ be a morphism of schemes. 
It follows from Example~\ref{example-HHX} that  $\cHH^i(X/S) = 0$ for $i < 0$ and 
$\cHH^0(X/S) \simeq \rR^0f_*\cO_X$. 
If $f \colon X \to S$ is a flat proper surjective morphism with geometrically reduced and connected fibers, then the  morphism $\cO_S \to \rR^0f_*\cO_X$ is an isomorphism. 
It follows that in this case, $\Dperf(X)$ is a connected $S$-linear category. 
\end{example} 

The above example can sometimes be leveraged to deduce connectivity of a semiorthogonal component of $\Dperf(X)$. 
Recall that if $\cC \to \Dperf(X)$ is the embedding of an $S$-linear semiorthogonal component, then there is a restriction morphism 
\begin{equation*}
\cHH^*(X/S) \to \cHH^*(\cC/S). 
\end{equation*} 
Kuznetsov~\cite{kuznetsov-heights} introduced a general method for controlling the cocone of this morphism, which is particularly effective when $\cC$ is defined as the orthogonal to an exceptional collection. In \cite{kuznetsov-heights} everything is done relative to a base field, but the arguments work similarly over a base ring, which is the case we will need. 
We recall the result below after introducing some notation. 

Let $f \colon X \to S$ be a smooth proper morphism of schemes with $S = \Spec(A)$ affine. 
Let $E_1, \dots, E_n \in \Dperf(X)$ be a relative exceptional collection, i.e. 
$\Ext^\bullet_A(E_i, E_j) = 0$ for $i > j$ and $\Ext^\bullet_A(E_i, E_i) = A[0]$ for all $i$. 
Then there is an $S$-linear semiorthogonal decomposition 
\begin{equation} 
\label{C-rel-Ei}
\Dperf(X) = \llangle \cC,  f^*\Dperf(S) \otimes E_1, \dots, f^*\Dperf(S) \otimes E_n \rrangle. 
\end{equation} 
The \emph{pseudoheight} of the collection $E_1, \dots, E_n$ is 
\begin{equation*}
\ph(E_1, \dots, E_n) = \min_{1 \leq a_0 < a_1 < \cdots < a_p \leq n} 
\Big( \re(E_{a_0}, E_{a_1}) + \cdots + \re(E_{a_{p-1}}, E_{a_p})  + 
\re(E_{a_p}, \rS^{-1}(E_{a_0})) - p  
\Big) 
\end{equation*} 
where $\re(F,F') = \min \set{ k \st \Ext^k_A(E, E') \neq 0 }$ (defined to be $+\infty$ if $\Ext^\bullet_A(E, E') = 0$) and 
$\rS^{-1}(F) = F \otimes \omega_{X/S}^{-1}[-\dim(X/S)]$ is the inverse of the relative Serre functor. 

\begin{proposition}[\cite{kuznetsov-heights}]
\label{proposition-heights}
Let $f \colon X \to S$ be a smooth proper morphism of schemes with $S$ affine, 
let $E_1, \dots, E_n \in \Dperf(X)$ be a relative exceptional collection, and let $\cC$ be defined by the semiorthogonal decomposition~\eqref{C-rel-Ei}. 
Then the restriction morphism $\cHH^i(X/S) \to \cHH^i(\cC/S)$ is an isomorphism for $i \leq \ph(E_1, \dots, E_n)-2$ and an injection for $i = \ph(E_1, \dots, E_n)-1$. 
\end{proposition} 

Using this, we can give a simple criterion for connectedness of a semiorthogonal component. 

\begin{corollary}
\label{corollary-connected} 
Let $f \colon X \to S$ be a smooth proper surjective morphism of schemes with geometrically connected fibers. 
Let $E_1, \dots, E_n$ be finite locally free sheaves which form a relative exceptional collection in $\Dperf(X)$.
Let $\cC$ be defined by the semiorthogonal decomposition~\eqref{C-rel-Ei}. 
If the relative dimension satisfies $\dim(X/S) \geq n+1$, then $\cC$ is a connected $S$-linear category. 
\end{corollary} 

\begin{proof}
All of our assumptions are preserved by base change along a morphism $T \to S$, so we may assume $S$ is affine and must prove $\HH^i(\cC/S) = 0$ for $i < 0$ and $\cO_S \xrightarrow{\sim} \cHH^0(\cC/S)$ is an isomorphism. 
It follows from the definitions and the assumption that the $E_i$ are locally free sheaves that 
$\ph(E_1, \dots, E_n) \geq \dim(X/S) - n + 1$. 
Therefore, if $\dim(X/S) \geq n+1$ then Proposition~\ref{proposition-heights} shows the map $\cHH^i(X/S) \to \cHH^i(\cC/S)$ is an isomorphism for $i \leq 0$. 
By Example~\ref{example-scheme-connected} this finishes the proof. 
\end{proof}

Corollary~\ref{corollary-connected} implies connectedness of many Kuznetsov components, 
including those of Fano threefolds of Picard number $1$ or of cubic fourfolds. 
Let us spell out explicitly how this verifies the hypotheses of Theorem~\ref{theorem-Z} 
in one example.

\begin{example}
Let $\cY \to S$ be a smooth family of Fano threefolds of Picard number $1$ and index $2$, equipped with a line bundle $\cO_\cY(1)$ whose restriction to each geometric fiber $\cY_s$ is the ample generator of $\Pic(\cY_s)$. 
Then $\cO_\cY, \cO_\cY(1)$ is a relative exceptional collection (in the sense of \cite[\S3.3]{stability-families}, cf.~\cite{samokhin}), and the $S$-linear category $\Ku(\cY) \subset \Dperf(\cY)$ defined by 
\begin{equation*}
\Dperf(\cY) = \llangle \Ku(\cY), f^*\Dperf(S), f^*\Dperf(S) \otimes  \cO_{\cY}(1)   \rrangle 
\end{equation*} 
is smooth, proper, and connected over $S$, and satisfies $\Ku(\cY)_s \simeq \Ku(\cY_s)$ for every geometric point $s \in S$. 
Indeed, $\cO_{\cY}, \cO_{\cY}(1)$ is a relative exceptional collection because this is so fiberwise. 
The $S$-linear category $\Ku(\cY)$ is smooth and proper as $\cY \to S$ is so, 
connected by Corollary~\ref{corollary-connected}, 
and by base change has as fibers the Kuznetsov components of the fibers of $\cY \to S$. 
\end{example} 

%%%%%%%%%%%%%%%%%%%%%%%%%%%%%%%%%%%%%%%%%%%%%%%%%%%%%%

\section{Group actions on categories} 
\label{section-group-actions-cats} 

Throughout this section, $G$ denotes a finite group. 
There are two notions of an action of  $G$ on a triangulated category $\cC$ appearing in the literature. The first one, often considered in mirror symmetry, is simply a group homomorphism $\phi$ from $G$ to the group of autoequivalences, considered up to natural transformations. A second, finer notion was originally introduced by Deligne in \cite{Deligne:groupedestresses} and requires a choice of natural transformations
\[ \phi(g) \circ \phi(g') \cong \phi(g \cdot g')
\]
compatible with triple compositions. This finer notion is necessary for the definition of $G$-equivariant objects; we refer to \cite{beckmann-oberdieck} for a recent account, including of the obstruction to lifting the former notion to the latter.

In this section, we consider actions on the homotopy category $\rh \ccD$ of an $\infty$-category $\ccD$ that lift to an action on $\ccD$. In this context, the first and second notion above correspond to 1-categorical and 2-categorical group actions on $\ccD$ . 

First in \S\ref{subsection-group-actions} we discuss some generalities on group actions, 
focusing on obstructions to lifting 1-categorical group actions to $\infty$-categorical actions. 
In \S\ref{subsection-group-action-linear-cats} we 
specialize to the case of group actions on linear categories, 
and show that  for connected linear categories over a base scheme 
there is a single obstruction, 
whose vanishing is an open condition in the \'{e}tale topology of the base (Proposition~\ref{proposition-etale-vanishing}). 
Finally, in \S\ref{subsection-invariant-categories} we study the category of (co)invariants for a group action on a linear category; this is the $\infty$-categorical analogue of Elagin's notion \cite{elagin} of the triangulated category of $G$-equivariant objects. In particular, we show that if the order of the group is invertible on the base scheme, then invariants commute with base change (Lemma~\ref{lemma-invariants-bc}) and preserve the property of being smooth and proper (Proposition~\ref{proposition-invariants-smooth-proper}).

\subsection{Group actions $\infty$-categorically}
\label{subsection-group-actions} 
We freely use the language of $\infty$-categories, as developed in \cite{HTT}. 
We often think of $\infty$-groupoids interchangeably as topological spaces, under 
the standard correspondence (given by passage to geometric realizations and singular simplicial sets). 

We denote by $BG$ the classifying space of $G$.  
When regarded as an $\infty$-groupoid, $BG$ is the nerve of the ordinary category with a single object whose endomorphisms are $G$. 
We write $* \in BG$ for the unique object. 

\begin{definition}
\label{definition-group-action}
Let $\ccD$ be an $\infty$-category, and let $X \in \ccD$ be an object. 
An \emph{action} of $G$ on $X$ is a functor $\phi \colon BG \to \ccD$ 
such that $\phi(*) = X$. 
\end{definition} 

Let us relate this $\infty$-categorical definition to some more classical notions. 

\subsubsection{1-categorical actions} 
\label{1-categorical} 
Suppose that $\ccD$ is an $1$-category. 
For such a category, we denote by $N(\ccD)$ its nerve, which is an $\infty$-category. 
Let $(BG)_1$ denote the $1$-category with a single object $*$ whose endomorphisms are $G$, 
so that $N((BG)_1) = BG$ by definition. 
Recall \cite[\href{https://kerodon.net/tag/002Y}{Tag 002Y}]{kerodon} 
the nerve construction induces a bijection 
\begin{equation*}
\Hom((BG)_1, \ccD) \cong \Hom(BG, N(\ccD)), 
\end{equation*} 
where the left side is the set of all functors of $1$-categories $(BG)_1 \to \ccD$ and 
the right is the set of all functors of $\infty$-categories $BG \to N(\ccD)$. 
Note that a functor $(BG)_1 \to \ccD$ taking $*$ to an object $X \in \ccD$ is 
equivalent to the data of a homomorphism $G \to \Aut X$ to the group of automorphisms of $X$; 
we call this a \emph{1-categorical action} of $G$ on $X$. 
Thus under the nerve construction, Definition~\ref{definition-group-action} recovers the notion of a $1$-categorical action. 

\subsubsection{2-categorical actions} 
\label{2-categorical}
Suppose that $\ccD$ is $(2,1)$-category, i.e. a $2$-category whose $2$-morphisms are all invertible. 
For such a category, we denote by $N(\ccD)$ its Duskin nerve \cite[\href{https://kerodon.net/tag/009T}{Tag 009T}]{kerodon}, which is an $\infty$-category \cite[\href{https://kerodon.net/tag/00AC}{Tag 00AC}]{kerodon}. 
If $\ccD$ is a $1$-category, regarded as a $2$-category with only identity $2$-morphisms, then 
the Duskin nerve is identified with the usual nerve, so the notation $N(\ccD)$ is unambiguous. 
By \cite[\href{https://kerodon.net/tag/00AU}{Tag 00AU}]{kerodon} the Duskin nerve construction induces a bijection 
\begin{equation*}
\Hom_{\mathrm{ULax}}((BG)_1, \ccD) \cong \Hom(BG, N(\ccD)), 
\end{equation*} 
where the left side denotes the set of all strictly unitary lax functors $(BG)_1 \to \ccD$ 
(in the sense of \cite[\href{https://kerodon.net/tag/008R}{Tag 008R}]{kerodon}) and we regard $(BG)_1$ as a $2$-category with only identity $2$-morphisms.  
For simplicity let us assume $\ccD$ is a strict $(2,1)$-category; then concretely, a strictly unitary lax functor $(BG)_1 \to \ccD$ taking $*$ to $X$ amounts to the following data, which we call a \emph{2-categorical action}  of $G$ on $X$: 
\begin{itemize}
\item For every $g \in G$, a $1$-morphism $\phi(g) \colon X \to X$, such that $\phi(1) = \id_X$. 
\item For every pair $g,f \in G$, a $2$-morphism $\mu_{g,f} \colon \phi(g) \circ \phi(f) \Rightarrow \phi(gf)$, such that 
the diagram 
\begin{equation*}
\xymatrix{
\phi(h) \circ \phi(g) \circ \phi(f) \ar@2{->}[rr]^{\, \mu_{h,g} \phi(f) \,} \ar@2{->}[d]_{\phi(h)\mu_{g,f}} &&  \phi(hg) \circ \phi(f) \ar@2{->}[d]^{\mu_{hg,f}} \\ 
\phi(h) \circ \phi(gf) \ar@2{->}[rr]^{\mu_{h, gf}} && \phi(hgf) 
} 
\end{equation*} 
is commutative. 
\end{itemize} 

\subsubsection{Obstructions to $\infty$-actions} 
\label{subsection-obstructions} 
Now suppose $\ccD$ is an $\infty$-category. 
For an object $X \in \ccD$, let $\Aut X$ denote the space of automorphisms  
and let $B\Aut X$ denote its classifying space. 
As the sub-$\infty$-groupoid of $\ccD$ spanned by $X$ is equivalent to $B\Aut X$, 
a $G$-action on $X$ is tantamount to 
the data of a functor $\psi \colon BG \to B\Aut X$. 
We say two $G$-actions $\phi, \phi' \colon BG \to \ccD$ are \emph{equivalent} if there is an equivalence $\phi \simeq \phi'$ of functors. By the previous remark, equivalence classes of $G$-actions on an object $X \in \ccD$ are in bijection with the set 
of homotopy classes of maps from $BG$ to $B\Aut X$. 

This perspective is useful for building $G$-actions. 
Namely, consider the Postnikov tower   
\begin{equation*}
\cdots \to \tau_{\leq 2}( B \Aut X ) \to 
\tau_{\leq 1} (B \Aut X)  \to \tau_{\leq 0}( B \Aut X) = * 
\end{equation*} 
of $B \Aut X$. 
So $\pi_{i}(\tau_{\leq n}(B \Aut X)) =0$ for $i > n$,  
the map $\tau_{\leq n} (B \Aut X) \to \tau_{\leq n-1}(B \Aut X)$ is a fibration with fiber $K(\pi_n(B \Aut X), n)$, and 
$B \Aut X \simeq \lim \tau_{\leq n}(B \Aut X)$.  
(Here and below, we typically suppress basepoints when dealing with homotopy groups.) 
Note that there is an isomorphism 
\begin{equation}
\label{pi-BAut}
\pi_n(B \Aut X) \cong \pi_{n-1}( \Aut X) ; 
\end{equation} 
in particular, it follows $\tau_{\leq 1} (B \Aut X) \simeq B \, \pi_0(\Aut X)$. 

Now suppose we are given a map $BG \to \tau_{\leq 1} (B \Aut X)$. 
By the preceding observation, such a map corresponds via taking $\pi_1$ to a group homomorphism $\phi_1 \colon G \to \pi_0(\Aut X)$. 
We call $\phi_1$ a \emph{$1$-categorical action} of $G$ on $X$, 
because when $X$ is regarded as an object of the homotopy category $\rh \ccD$ 
then $\phi$ is precisely a $1$-categorical action in the sense of \S\ref{1-categorical}. 
We say a $G$-action $\phi \colon BG \to B \Aut X$ is an \emph{$\infty$-lift of $\phi_1$} if the composition 
\begin{equation*}
BG \xrightarrow{\phi} B\Aut X \to \tau_{\leq 1} (B \Aut X)
\end{equation*}
recovers $\phi_1$  (upon taking $\pi_1$).  
It is also convenient to study an intermediate notion. 
Namely, we call a map $\phi_n \colon BG \to \tau_{\leq n}(B \Aut X)$ an \emph{$n$-lift} of $\phi_1$ if the composition 
\begin{equation*}
BG \xrightarrow{\phi_n} \tau_{\leq n}( B\Aut X) \to \tau_{\leq 1} (B \Aut X)
\end{equation*} 
recovers $\phi_1$. 
We say two $n$-lifts are equivalent if they have the same homotopy class. 

The $n$-lifts of $\phi_1$ can be studied via obstruction theory; 
below we spell out the simplest case of $2$-lifts. 
Note that $\pi_0(\Aut X)$ acts on $\pi_i(\Aut X)$ for $i \geq 1$ via 
conjugation; under the identifications $\pi_0(\Aut X) \cong \pi_1(B\Aut X)$ 
and $\pi_i(\Aut X) \cong \pi_{i+1}(B \Aut X)$, this is the usual action of the 
fundamental group on higher homotopy groups. 
In particular, a $1$-categorical action $\phi_1 \colon G \to  \pi_0( \Aut X )$ 
also induces an action of $G$ on $\pi_i(\Aut X)$ for $i \geq 1$. 

\begin{lemma}
\label{lemma-infinity-action}
Let $\ccD$ be an $\infty$-category, let $X \in \ccD$ be an object, and 
let $\phi_1 \colon G \to \pi_0(\Aut X)$ be a $1$-categorical action of $G$ on $X$.  
Regarding $\pi_1(\Aut X)$ as a local system on $BG$  
via the action of $G$ described above, 
then there is a canonical obstruction class 
\begin{equation*} 
\ob(\phi_1) \in \rH^3(BG, \pi_1(\Aut X)) 
\end{equation*} 
such that a $2$-lift of $\phi_1$ exists if and only if $\ob(\phi_1) = 0$, and 
in this case the set of equivalence classes of $2$-lifts is a 
$\rH^2(BG, \pi_1(\Aut X))$-torsor. 
If $\pi_i(\Aut X) = 0$ for $i \geq 2$, then the same conclusion 
holds for $\infty$-lifts. 
\end{lemma} 

\begin{proof}
The first claim holds by standard obstruction theory. 
If $\pi_i(\Aut X) = 0$ for $i \geq 2$, then by~\eqref{pi-BAut} we have 
$B \Aut X \simeq \tau_{\leq 2}(B \Aut X)$, so $\infty$-lifts are the same 
as $2$-lifts. 
\end{proof} 

\subsection{Group actions on linear categories} 
\label{subsection-group-action-linear-cats}
Let $S$ be a base scheme. 
Recall that the collection of all $S$-linear categories (with morphisms between them the exact $S$-linear functors) can be organized into an $\infty$-category $\Cat_S$ \cite{NCHPD}. 
Thus, using the formalism of \S\ref{subsection-group-actions} we can make sense of $G$-actions on 
$S$-linear categories. 

\subsubsection{Obstructions in terms of Hochschild cohomology} 
Our main observation is that when negative Hochschild cohomology vanishes, then it is 
easy to classify $\infty$-lifts of $1$-categorical $G$-actions. 
For this, we need a preliminary lemma. 
If $\cC$ is an $S$-linear category, to emphasize the dependence on the $S$-linear structure we write $\Aut(\cC/S)$ for the space of $S$-linear autoequivalences of~$\cC$, i.e. the automorphism space of $\cC$ as an object of $\Cat_S$. 
Note also that $\HH^0(\cC/S) = \rH^0(\cHom_S(\id_\cC, \id_\cC))$ has a $\Gamma(S, \cO_S)$-algebra structure, so we may 
consider the group of units $\HH^0(\cC/S)^{\times}$.

\begin{lemma}
\label{lemma-piAut}
There are natural group isomorphisms 
\begin{equation*}
\pi_i(\Aut(\cC/S)) 
\cong 
\begin{cases}
\HH^0(\cC/S)^{\times}  & \text{if }~ i = 1 ,  \\ 
\HH^{1-i}(\cC/S) & \text{if }~ i \geq 2 . 
\end{cases} 
\end{equation*} 
\end{lemma} 

\begin{proof}
Note that in general, if $\ccD$ is an $\infty$-groupoid and $X \in \ccD$ is an object, then 
there is an isomorphism 
\begin{equation*}
\pi_i(\ccD, X) \cong 
\pi_{i-1}(\Aut X) 
\quad \text{for} \quad i \geq 1,
\end{equation*} 
where the left side denotes the homotopy 
group of $\ccD$ (thought of as a topological space) based at the point $X$. 
Further, suppose that $\ccD \hookrightarrow \ccD'$ is the sub-$\infty$-groupoid 
spanned by some collection of objects in an $\infty$-category $\ccD'$. 
If $\Map_{\ccD'}(X, X)$ denotes the space of endomorphisms of $X$ as an object of $\ccD'$, 
then $\pi_0(\Map_{\ccD'}(X, X))$ is a monoid, whose group of invertible elements we denote 
by $\pi_0(\Map_{\ccD'}(X, X))^{\times}$. 
Then there are isomorphisms 
\begin{equation*}
\pi_i(\Aut X) 
 \cong 
\begin{cases}
\pi_0(\Map_{\ccD'}(X, X))^{\times} & \text{if }~ i = 0, \\ 
\pi_i(\Map_{\ccD'}(X,X)) & \text{if }~ i \geq 1. 
\end{cases}
\end{equation*} 
Indeed, for $i = 0$ this holds by the definition of $\Aut X$, while for $i \geq 1$ this holds because $k$-morphisms in $\ccD'$ are invertible (and hence coincide with $k$-morphisms in $\ccD$) for $k \geq 2$. 

As $\Aut(\cC/S)$ is the sub-$\infty$-groupoid spanned by the autoequivalences in  
the $\infty$-category $\Fun_S(\cC, \cC)$, combining the above observations shows 
\begin{equation*}
\pi_i(\Aut(\cC/S), \id_{\cC}) \cong 
\begin{cases}
\pi_0(\Map_{\Fun_S(\cC,\cC)}(\id_\cC, \id_\cC))^{\times} & \text{if }~ i = 1, \\ 
\pi_{i-1}(\Map_{\Fun_S(\cC, \cC)}(\id_\cC, \id_\cC)) & \text{if }~ i \geq 2. 
\end{cases}
\end{equation*} 
By the characterizing property of mapping objects~\eqref{mapping-object} and 
the definition~\eqref{HH-definition} of Hochschild cohomology, we have 
\begin{equation*}
\Map_{\Fun_S(\cC, \cC)}(\id_\cC, \id_\cC) \simeq 
\Map_{\Dqc(S)}(\cO_S, \cHH^*(\cC/S)). 
\end{equation*} 
Taking homotopy groups, we conclude 
\begin{equation*}
\pi_i \Map_{\Fun_S(\cC, \cC)}(\id_\cC, \id_\cC) \simeq \Ext^{-i}(\cO_S, \cHH^*(\cC/S)) = 
\HH^{-i}(\cC/S). 
\end{equation*} 
All together, this proves the claimed formula for $\pi_i(\Aut(\cC/S))$. 
\end{proof} 

Note that via conjugation $\pi_0(\Aut(\cC/S))$ acts on the $\Gamma(S, \cO_S)$-algebra $\HH^0(\cC/S)$, 
and hence so does $G$ for any $1$-categorical action $\phi_1 \colon G \to  \pi_0( \Aut (\cC/S) )$. 
Combining Lemma~\ref{lemma-infinity-action} and Lemma~\ref{lemma-piAut} gives the following. 

\begin{corollary}
\label{corollary-obstruction-HH}
Let $\cC$ be an $S$-linear category, and 
let $\phi_1 \colon G \to \pi_0( \Aut(\cC/S) )$ be a $1$-categorical action of $G$ on $\cC$. 
Regarding $\HH^0(\cC/S)^{\times}$ as a local system on $BG$ via the action of $G$ 
described above, 
then there is a canonical obstruction class 
\begin{equation*}
\ob(\phi_1) \in \rH^3(BG, \HH^0(\cC/S)^{\times}) 
\end{equation*}   
such that a $2$-lift of $\phi_1$ exists if and only if $\ob(\phi_1) = 0$, 
and in this case the set of equivalence classes of $2$-lifts is a $\rH^2(BG, \HH^0(\cC/S)^{\times})$-torsor. 
If $\HH^i(\cC/S) = 0$ for $i < 0$, then the same conclusion holds for $\infty$-lifts. 
\end{corollary} 

\begin{remark}
\label{remark-connected-HH0}
Suppose $\cC$ is an $S$-linear category 
such that the map $\Gamma(S, \cO_S) \to \HH^0(\cC/S)$ is an isomorphism; 
this holds if $\cC$ is a connected $S$-linear category, such as  
$\cC = \Dperf(X)$ 
for a morphism $X \to S$ with assumptions as in Example~\ref{example-scheme-connected}. 
Then $\pi_0(\Aut(\cC/S))$ acts trivially on $\HH^0(\cC/S)$, because  
the conjugation action fixes $\id \colon \id_{\cC} \to \id_{\cC}$ 
and the isomorphism $\Gamma(S, \cO_S) \to \HH^0(\cC/S)$ takes $1$ to $\id_{\cC}$. 
Hence if this condition holds in the setup of Corollary~\ref{corollary-obstruction-HH}, 
then the obstruction class lies in $\rH^3(BG, \Gamma(S,\cO_S)^{\times})$, where $\Gamma(S,\cO_S)^{\times}$ is 
a constant local system on $BG$. 
\end{remark} 

\begin{remark}
\label{remark-AutC}
Let us make some comments about $\Aut(\cC/S)$ for an $S$-linear category $\cC$. 
\begin{enumerate}
\item \label{remark-AutC-FM}
Suppose $\cC = \Dperf(X)$ where $X \to S$ is a smooth proper morphism of schemes.  
Then by \cite{bzfn} there is an equivalence $\Dperf(X \times_S X) \simeq \Fun_{S}(\Dperf(X), \Dperf(X))$ 
that takes $K \in \Dperf(X \times_S X)$ to the Fourier--Mukai functor $\Phi_K = \pr_{2*}(\pr_1^*(-) \otimes K)$. 
Therefore, $\Aut(\Dperf(X)/S)$ is equivalent to the sub-$\infty$-groupoid of $\Dperf(X \times_S X)$ 
spanned by objects $K$ such that $\Phi_K$ is an equivalence. 

\item Suppose $S = \Spec A$ is affine.  
The homotopy category $\rh \cC$ is a triangulated category which is $A$-linear in the sense that it is enriched in $A$-modules. 
Let $\pi_0(\Aut(\rh \cC/A))$ be the group of $A$-linear exact autoequivalences of $\rh\cC$ modulo isomorphisms of functors. 
(Classically, this group may be denoted $\Aut(\rh \cC/A)$, but it is more consistent with our notation above to rather denote by $\Aut(\rh \cC/A)$ the groupoid 
whose objects are $A$-linear exact autoequivalences and whose morphisms are isomorphisms of functors.) 
There is a natural group homomorphism 
\begin{equation*}
\pi_0(\Aut(\cC/A)) \to \pi_0(\Aut(\rh \cC/A)). 
\end{equation*} 
The question of when this map is an isomorphism is interesting and difficult. 
In the case where $A=k$ is a field and $\cC = \Dperf(X)$ for a smooth projective variety $X$ over $k$, the answer is positive; 
indeed, this follows from~\eqref{remark-AutC-FM} together with the existence and uniqueness of Fourier--Mukai kernels for triangulated equivalences \cite[Theorem 2.2]{orlov-K3}. 
\end{enumerate} 
\end{remark} 

\begin{remark}
Our arguments have similar consequences for group actions on $1$-categories. 
Let $A$ be a ring. 
We use the term \emph{classical $A$-linear category} to mean a $1$-category which is enriched in 
$A$-modules. 
Let $\Catcl_A$ be the strict $(2,1)$-category with objects the classical $A$-linear categories, 
$1$-morphisms the $A$-linear functors, and $2$-morphisms the isomorphisms of functors. 
For any $\cC \in \Catcl_A$, we can consider the notion of a $2$-categorical action 
on $\cC$ (in the sense of~\S\ref{2-categorical}). 
Any such action induces a  homomorphism $\phi_1 \colon G \to \pi_0(\Aut(\cC/A))$, 
where $\pi_0(\Aut(\cC/A))$ denotes the group of $A$-linear autoequivalences modulo isomorphisms of functors. 

Conversely, suppose we are given a homomorphism $\phi_1 \colon G \to \pi_0(\Aut(\cC/A))$, and we 
want to understand when it lifts to a $2$-categorical action. 
By analogy with the case of linear categories, 
define the $A$-algebra $\HH^0(\cC/A) \coloneqq \Hom(\id_{\cC}, \id_{\cC})$. 
Then there is a canonical obstruction class $\ob(\phi_1) \in \rH^3(BG, \HH^0(\cC/A)^{\times})$ 
such that a $2$-categorical action lifting $\phi_1$ exists if and only if $\ob(\phi_1) = 0$, 
and in this case the set of equivalence classes of $\infty$-lifts is a $\rH^2(BG, \HH^0(\cC/S)^{\times})$-torsor. 
Indeed, 
the Duskin nerve $N(\Catcl_A)$ is an $\infty$-category, $\cC$ can be thought 
of as an object of $N(\Catcl_A)$, and its corresponding automorphism space 
$\Aut(\cC/A)$ has $\pi_0$ as described above,  
$\pi_1(\Aut(\cC/A)) \cong \HH^0(\cC/A)^{\times}$, and vanishing higher homotopy groups; 
therefore, the claim follows from Lemma~\ref{lemma-infinity-action} and the correspondence 
between $\infty$-categorical actions on $\cC \in N(\Catcl_A)$ and $2$-categorical actions on $\cC \in \Catcl_A$ described in~\S\ref{2-categorical}. 

In the case where $\HH^0(\cC/A) \cong A$ and $A = \bC$ is the field of complex numbers, 
this obstruction to $2$-categorical actions was proved in \cite[Theorem 2.1]{beckmann-oberdieck} 
by a hands-on cocycle argument; the advantage of our proof is that it is more conceptual and generalizes to $\infty$-categorical actions. 
We also refer to \cite[\S3.6]{beckmann-oberdieck} for some simple examples where this obstruction is nontrivial. 
\end{remark}

\subsubsection{Base change of $G$-actions} 
Let $\cC$ be an $S$-linear category. 
As discussed in \S\ref{section-linear-categories}, for any morphism of schemes $T \to S$ we can form the base change 
category $\cC_T$ which is linear over~$T$. 
Formation of base change is functorial, i.e. gives a functor $\Cat_S \to \Cat_T$; 
therefore, any $G$-action $\phi \colon BG \to \Cat_{S}$ on $\cC$ induces a \emph{base changed $G$-action} 
$\phi_T \colon BG \to \Cat_T$ on $\cC_T$ by composition with this functor. 
Similarly, any $1$-categorical action $\phi_1 \colon G \to \pi_0(\Aut(\cC/S))$ of $G$ on $\cC$ 
induces a \emph{base changed $1$-categorical action} $(\phi_1)_T \colon G \to \pi_0(\Aut(\cC_T/T))$ on $\cC_T$, by 
composition with $\pi_0$ of the map $\Aut(\cC/S) \to \Aut(\cC_T/T)$. 
Note also that the base change map for group actions takes $n$-lifts of $\phi_1$ (in the sense of \S\ref{subsection-obstructions})  to $n$-lifts of $(\phi_1)_T$. 

By Remark~\ref{remark-HH-bc}, we also have a natural ring map $\HH^0(\cC/S) \to \HH^0(\cC_T/T)$, 
which is easily seen to be compatible with the actions of $\pi_0(\Aut(\cC/S))$ and $\pi_0(\Aut(\cC_T/T))$ under the map 
$\pi_0(\Aut(\cC/S)) \to \pi_0(\Aut(\cC_T/T))$. 

The next lemma follows by unwinding our construction of $\ob(\phi_1)$ and using functoriality of all the constructions involved. 

\begin{lemma}
\label{lemma-ob-functorial}
Let $\cC$ be an $S$-linear category, and let $\phi_1 \colon G \to \pi_0(\Aut(\cC/S))$ be a $1$-categorical action of $G$ on $\cC$. 
Then the obstruction class of Corollary~\ref{corollary-obstruction-HH} is functorial under base change in the sense that for any morphism $T \to S$, 
the natural map 
\begin{equation*}
\rH^3(BG, \HH^0(\cC/S)^{\times}) \to \rH^3(BG, \HH^0(\cC_T/T)^{\times}) 
\end{equation*} 
takes $\ob(\phi_1)$ to $\ob((\phi_1)_T)$. 
Moreover, the map from the set of equivalence classes of $2$-lifts of $\phi_1$ to 
the set of those of $(\phi_1)_T$ is compatible with the torsor structures under the map 
$\rH^2(BG, \HH^0(\cC/S)^{\times}) \to \rH^2(BG, \HH^0(\cC_T/T)^{\times})$. 
\end{lemma} 

\subsubsection{Vanishing of obstructions on \'{e}tale neighborhoods} 

Recall that a ring $A$ is called a \emph{Grothendieck ring} if it is noetherian and for every  $\mathfrak{p} \in \Spec A$ the completion $A_{\mathfrak{p}} \to \widehat{A_{\mathfrak{p}}}$ of the local ring at $\mathfrak{p}$ is a regular map of rings  \cite[\href{https://stacks.math.columbia.edu/tag/07GG}{Tag 07GG}]{stacks-project}. 
(This is often called a \emph{G-ring}, but we will not use that terminology to avoid confusion with the group $G$.)
A scheme $S$ is called a \emph{Grothendieck scheme} if for every open affine $U \subset S$ the ring $\cO_S(U)$ is a Grothendieck ring. 
This is a very mild condition, which includes all excellent schemes. 

\begin{proposition}
\label{proposition-etale-vanishing}
Let $S$ be a Grothendieck scheme and let $\cC$ be a connected $S$-linear category. 
Let $\phi_1 \colon G \to \pi_0(\Aut(\cC/S))$ be a $1$-categorical action of $G$ on $\cC$. 
Let $s \in S$ be a point such that 
the characteristic of the residue field $\kappa(s)$ is prime to the order of $G$, and 
the obstruction $\ob((\phi_1)_s) \in \rH^3(BG, \kappa(s)^{\times})$ vanishes.
Then there exists an \'{e}tale neighborhood $U \to S$ of $s$ such that 
the obstruction $\ob((\phi_1)_U) \in \rH^3(BG, \Gamma(U, \cO_U)^{\times})$ vanishes, 
and thus the set of equivalence classes of $\infty$-lifts of $(\phi_1)_U$ is a nonempty $\rH^2(BG, \Gamma(U, \cO_U)^{\times})$-torsor. 
\end{proposition} 

\begin{proof}
As $\cC$ is connected, 
the natural map $\Gamma(T, \cO_T) \to \HH^0(\cC_T/T)$ is an isomorphism 
for any $T \to S$, and 
$G$ acts trivially on $\HH^0(\cC_T/T)^{\times}$ (Remark~\ref{remark-connected-HH0}); 
in particular, the obstructions indeed lie in the stated groups. 
Further, as the claim is local on $S$, we may assume $S = \Spec A$ is affine. 
By functoriality of the obstruction under base change (Lemma~\ref{lemma-ob-functorial}), the result is then a consequence of the following lemma. 
\end{proof} 

\begin{lemma}
\label{lemma-etale-vanishing-ring}
Let $A$ be a Grothendieck ring. 
Let $\alpha \in \rH^n(BG, A^{\times})$ where $n \geq 1$ and $A^{\times}$ has the trivial $G$-action. 
Let $\fp \in \Spec A$ be a point such that the characteristic of $\kappa(\fp)$ is prime to the order of $G$, and 
such that $\alpha$ maps to zero under $\rH^n(BG, A^{\times}) \to \rH^n(BG, \kappa(\fp)^{\times})$. 
Then there exists an affine \'{e}tale neighborhood $\Spec(B) \to \Spec(A)$ of $\fp$ such that $\alpha$ 
maps to zero under $\rH^n(BG, A^{\times}) \to \rH^n(BG, B^{\times})$. 
\end{lemma} 

Before proving the result in general, we handle the case of complete local rings: 

\begin{lemma}
\label{lemma-BG-clr}
Let $A$ be a complete local ring with residue field $\kappa$ 
of characteristic prime to the order of $G$. 
Then for $n \geq 1$ the map $\rH^n(BG, A^{\times}) \to \rH^n(BG, \kappa^{\times})$ is an isomorphism, 
where $A^{\times}$ and $\kappa^{\times}$ have the trivial $G$-actions.  
\end{lemma} 

\begin{proof}
First consider the case where $A$ is artinian. 
Then the maximal ideal satisfies $\fm_A^i = 0$ for some $i \geq 1$, so by considering 
the factorization $A = A/\fm_A^i \to A/\fm_A^{i-1} \to \cdots \to \kappa$, 
we reduce to proving the following claim: if $A \to B$ is a surjection of artinian local rings 
whose kernel $I$ is annihilated by $\fm_A$, then for $n \geq 1$ the map 
$\rH^n(BG, A^{\times}) \to \rH^n(BG, B^{\times})$ is an isomorphism. 
In this case, we have an exact sequence 
\begin{equation*}
0 \to 1 + I \to A^{\times} \to B^{\times} \to 0, 
\end{equation*}
and as an abelian group $1+I$ is isomorphic to the $\kappa$-vector space $I$; 
as $|G|$ is invertible in $\kappa$ this implies $\rH^n(BG, 1+I) = 0$ for $n \geq 1$, 
and hence the claim. 

Now consider the case of a general complete local ring $A$. 
Then $A^{\times} = \lim_i A_i^{\times}$ where $A_i = A/\fm_A^i$, 
so we have an exact sequence 
\begin{equation*}
0 \to \rR^1\lim_i \rH^{n-1}(BG, A_i^{\times}) \to \rH^n(BG, A^{\times}) 
\to \lim_i \rH^n(BG, A_i^{\times}) \to 0. 
\end{equation*} 
By the artinian case handled above, it suffices to show that the first term in this sequence vanishes. 
But if $n \geq 2$ then again by the artinian case the transition maps for the system $\rH^{n-1}(BG, A_i^{\times})$ are isomorphisms, while if $n = 1$ then the transition maps are surjective; in either case, the system is Mittag-Leffler and we conclude that their $\rR^1 \lim_i$ vanishes. 
\end{proof}

\begin{proof}[Proof of Lemma~\ref{lemma-etale-vanishing-ring}]
By assumption $\alpha$ dies in $\rH^n(BG, \kappa(\fp)^{\times})$, 
so by Lemma~\ref{lemma-BG-clr} it also dies in $\rH^n(BG, \widehat{A_{\fp}}{^{\times}})$. 
As $A$ is a Grothendieck ring, the composition $A \to A_{\fp} \to \widehat{A_{\fp}}$ is a regular map of noetherian rings, 
so by Popescu's theorem \cite[\href{https://stacks.math.columbia.edu/tag/07GC}{Tag 07GC}]{stacks-project} we can write $\widehat{A_{\fp}} = \colim A_i$ as a filtered colimit of smooth ring maps $A \to A_i$. 
Then $\rH^n(BG, \widehat{A_{\fp}}{^{\times}}) = \colim \rH^n(BG, A_i^{\times})$ 
as $\widehat{A_{\fp}}{^\times} = \colim A_i^{\times}$, and hence $\alpha$ must 
die in $\rH^n(BG, A_i^{\times})$ for some $i$. 
As $\Spec(A_i) \to \Spec(A)$ is smooth and its image contains $\fp$, 
we can find \'{e}tale neighborhood $\Spec(B) \to \Spec(A)$ of $\fp$ over which 
$\Spec(A_i) \to \Spec(A)$ has a section. 
In other words, the map $A \to B$ factors through $A \to A_i$, 
and therefore $\alpha$ dies in $\rH^n(BG, B^{\times})$. 
\end{proof} 

\begin{example}
\label{example-etale-vanishing}
In the conclusion of Lemma~\ref{lemma-etale-vanishing-ring}, ``\'{e}tale neighborhood'' cannot be replaced by ``Zariski neighborhood'' in general. 
Indeed, let $G = \bZ/m$ for $m \geq 2$,  
and let $A = \bC[x,x^{-1}]$. 
Note that $A^{\times}  \cong \bC^{\times} \oplus \bZ$, with $(a,b) \in \bC^{\times} \oplus \bZ$ corresponding to $ax^{b}$. 
If $n> 0$ is even, then $\rH^n(BG, \bC^{\times}) = 0$ and $\rH^{n}(BG, \bZ) = \bZ/m$. 
Hence every element of $\rH^n(BG, A^{\times}) = \bZ/m$ is killed by the 
map $\rH^n(BG, A^{\times}) \to \rH^n(BG, \bC^{\times})$ induced by any closed point $\Spec \bC \to \Spec A$. 
However, it is easy to see that for any $f \in A$, the map to the localization 
$A \to A_f$ induces a split injection $A^{\times} \to (A_f)^{\times}$ on groups of units, 
and hence $\rH^n(BG, A^{\times}) \to  \rH^n(BG, (A_f)^{\times})$ is also a split injection. 
It follows that any $0 \neq \alpha \in \rH^n(BG, A^{\times})$ dies after restriction to any closed point of $\Spec A$, but has nonzero image in $\rH^n(BG, \Gamma(U, \cO_U)^{\times})$ for any Zariski open $U \subset \Spec A$. 
\end{example} 

\subsection{Invariant categories} 
\label{subsection-invariant-categories}
The (co)invariants for a group action can be formulated in the $\infty$-categorical setting as follows. 

\begin{definition}
Let $\ccD$ be an $\infty$-category, let $X \in \ccD$ be an object, and let 
$\phi \colon BG \to \ccD$ be an action of $G$ on $X$. 
The \emph{$G$-invariants $X^G$} and \emph{$G$-coinvariants $X_G$} of the action $\phi$ are defined 
by 
\begin{equation*} 
X^G = \lim(\phi) \qquad \text{and} \qquad X_G = \colim(\phi) , 
\end{equation*} 
provided the displayed limit and colimit exist. 
\end{definition} 

We will be interested in the case where $\ccD = \Cat_S$ is the $\infty$-category of $S$-linear categories over some base scheme $S$, and $G$ acts on an $S$-linear category $\cC \in \Cat_S$. 
Note that $\Cat_S$ has all limits and colimits; 
see for instance~\cite[\S2.1]{akhil} where this result is explained for $\Cat_{\infty}^{\mathrm{st}}$, the $\infty$-category of small stable $\infty$-categories, and it similarly holds for $\Cat_S$. 
%(cf.~\cite[\S2.1]{akhil} where this result is explained for $\Cat^{\mathrm{st}}_{\infty}$ in place of $\Cat$), 
The $G$-invariants $\cC^G$ and coinvariants $\cC_G$ thus always exist in this situation. 
A basic example to keep in mind is that for a scheme $X$ with a $G$-action, $\Dperf(X)^G \simeq \Dperf([X/G])$ where $[X/G]$ is the quotient stack. 

\subsubsection{Base change of $G$-(co)invariants} 
The operations of taking $G$-coinvariants commutes with 
base change, while the same is true for $G$-invariants if the order of $G$ is invertible on the base scheme: 

\begin{lemma}
\label{lemma-invariants-bc}
Let $\cC$ be an $S$-linear category with a $G$-action, 
and let $T \to S$ be a morphism of schemes. 
\begin{enumerate}
\item \label{coinvariants-bc}
There is an equivalence $(\cC_{G})_T \simeq (\cC_T)_G$ of 
$T$-linear categories, where $(\cC_T)_G$ denotes the $G$-coinvariants for the induced $G$-action on the base change $\cC_T$. 
\item \label{invariants-bc} If $|G|$ is invertible on $S$, then there is an equivalence $(\cC^{G})_T \simeq (\cC_T)^G$ of $T$-linear categories. 
\end{enumerate}
\end{lemma} 

\begin{proof}
\eqref{coinvariants-bc} The base change functor $\Cat_S \to \Cat_T$, $\cC \mapsto \cC_T$ has a right adjoint, given by the functor which regards a $T$-linear category as an $S$-linear category via restriction along $\Dperf(S) \to \Dperf(T)$, and therefore commutes with colimits. 

\medskip \noindent
\eqref{invariants-bc} 
There is a canonical norm functor $\Nm \colon \cC_G \to \cC^G$, 
which under our assumption on $|G|$ is an equivalence; see 
\cite[Proposition 3.4]{HH} where this result is stated for $S$ a field, but 
the same proof works in general. 
Therefore, the claim reduces to~\eqref{coinvariants-bc} proved above. 
\end{proof}

\subsubsection{$G$-invariants of smooth proper categories} 
Passage to $G$-invariants preserves smooth and properness of a category, as long as $|G|$ is invertible on the base scheme: 

\begin{proposition}
\label{proposition-invariants-smooth-proper}
Let $\cC$ be a smooth proper $S$-linear category, where $|G|$ is 
invertible on $S$. 
Then the $G$-invariant category $\cC^{G}$ is also smooth and proper. 
\end{proposition} 

\begin{proof}
Properness of $\cC^G$ is the assertion that for any objects $E, F \in \cC^G$ their mapping object $\cHom_S(E,F) \in \Dqc(S)$ lies in $\Dperf(S)$ (see e.g.~\cite[Lemma 4.7]{NCHPD}). 
This mapping object can be described by the formula 
\begin{equation*}
\cHom_S(E,F) = \cHom_S(\Forg(E), \Forg(F))^G, 
\end{equation*} 
where $\Forg \colon \cC^G \to \cC$ is the forgetful functor, 
and the right side is the group invariants for the induced $G$-action 
on $\cHom_S(\Forg(E), \Forg(F))$  (see e.g. \cite[\S3.1]{HH}). 
The object $\cHom_S(\Forg(E), \Forg(F))$ is perfect by the properness of $\cC$. 
By our assumption on $|G|$ the object $\cHom_S(E,F)$ is a summand of $\cHom_S(\Forg(E), \Forg(F))$, hence also perfect. 
Indeed, more generally if $A \in \Dperf(S)$ is an object equipped with a $G$-action, then $A^G$ is a summand of $A$: if $\phi_g$ denotes the automorphism of $A$ corresponding to $g \in G$, then $\frac{1}{|G|}\sum_{g \in G} \phi_g \colon A \to A^G$ gives the splitting. 
%, hence also perfect. 

Smoothness of $\cC^G$ is the assertion that $\id_{\Ind(\cC^G)} \in \Fun_{S}(\Ind(\cC^G), \Ind(\cC^G))$ 
is a compact object 
(see~\cite[\S4]{NCHPD} for background on ind completions and smoothness of categories). 
To prove this, we use some results from \cite{HH}; we note that while results there are stated for categories linear over a field, 
they also hold relative to a base scheme $S$ by the same arguments. 
First we note that $\Ind(\cC^G) \simeq \Ind(\cC)^G$; indeed, by \cite[Proposition 3.4 and Lemma 3.5]{HH}, the assertion is equivalent to the analogous statement $\Ind(\cC_G) \simeq \Ind(\cC)_G$ for coinvariant categories, which holds by \cite[Lemma 2.3(1)]{HH}. 
%to the assertion for coinvariant categories 
%\cite[Lemma 3.6]{HH} we have $\cC^G \simeq (\Ind(\cC)^G)^c$. 
%On the other hand, by 
%$\Ind(\cC^G) \simeq \Ind(\cC)^G$. 
By \cite[Lemma 4.7]{HH} (or rather the corresponding result for presentable $S$-linear categories, see \cite[Remark 4.6]{HH}), we have an equivalence 
\begin{equation*}
\Fun_S(\Ind(\cC)^G, \Ind(\cC)^G) \simeq \Fun_S(\Ind(\cC), \Ind(\cC))^{G \times G} 
\end{equation*} 
which sends $\id_{\Ind(\cC)^G}$ to the functor 
$\bigoplus_{g \in G} \Ind(\phi_g) \colon \Ind(\cC) \to \Ind(\cC)$, where 
$\phi_g \colon \cC \to \cC$ is the equivalence corresponding to the action of $g \in G$ on $\cC$. 
By \cite[Lemma 3.7]{HH} the functor $\bigoplus_{g \in G} \Ind(\phi_g)$ is compact as an object of 
$\Fun_S(\Ind(\cC), \Ind(\cC))^{G \times G}$ if and only if it is compact as an object of 
$\Fun_S(\Ind(\cC), \Ind(\cC))$. 

The $S$-linear category $\cC$, being smooth and proper, 
is dualizable with dual the opposite category $\cC^{\op}$ (see \cite[Lemma 4.8]{NCHPD}), 
and the presentable $S$-linear category $\Ind(\cC)$ is dualizable with dual $\Ind(\cC^{\op})$  
(see \cite[Lemma 4.3]{NCHPD}). 
Thus we have equivalences 
\begin{align*}
\Fun_S(\Ind(\cC), \Ind(\cC)) & \simeq 
\Ind(\cC^{\op}) \otimes_{\Dqc(S)} \Ind(\cC) \\ 
& \simeq \Ind \left( \cC^{\op} \otimes_{\Dperf(S)} \cC \right) \\ 
& \simeq \Ind \left( \Fun_{S}(\cC, \cC) \right) 
\end{align*} 
where the second line holds by the definition of the tensor product of $S$-linear categories (see \cite[\S2.3]{NCHPD}). 
This shows that for a smooth proper $S$-linear category $\cC$, 
the compact objects of $\Fun_S(\Ind(\cC), \Ind(\cC))$ are precisely those in the image of 
$\Fun_S(\cC, \cC) \to \Fun_S(\Ind(\cC), \Ind(\cC))$. 

In particular, we see that the functor $\bigoplus_{g \in G} \Ind(\phi_g)$ from above is compact as an object of $\Fun_S(\Ind(\cC), \Ind(\cC))$, 
because it is the image of the object $\bigoplus_{g \in G} \phi_g \in \Fun_S(\cC, \cC)$. 
Hence $\cC^G$ is smooth. 
\end{proof} 

%%%%%%%%%%%%%%%%%%%%%%%%%%%%%%%%%%%%%%%%%%%%%%%%%%%%%%

\section{Enriques categories and their CY2 covers} 
\label{section-enriques} 

In this section we work over an algebraically closed field $k$ with $\characteristic(k) \neq 2$. 
We define Enriques and CY2 categories, explain the correspondence 
between them via residual $\bZ/2$-actions, and describe the examples of interest for this paper. 

\subsection{Definitions}
Recall that a smooth proper $k$-linear category $\cC$ admits a \emph{Serre functor}, i.e. an autoequivalence $\rS_{\cC}$ such that there are natural isomorphisms  
\begin{equation*}
\cHom_k(E, \rS_{\cC}(F)) \simeq \cHom_k(F, E)^{\svee} 
\end{equation*}
for $E, F \in \cC$.  
Recall also if $\cC \subset \Db(X)$ is a semiorthogonal component of the derived category of a variety, 
then $\cHom_k(E, F)$ coincides with the classical derived Hom complex $\mathrm{RHom}(E, F)$. 

\begin{definition}
Let $\cC$ be a smooth proper $k$-linear category. 
\begin{enumerate} 
\item $\cC$ is an \emph{Enriques category} if 
it is equipped with a $\bZ/2$-action whose generator $\tau$ is a nontrivial autoequivalence of $\cC$ 
satisfying $\rS_{\cC} \simeq \tau \circ [2]$. 
\item $\cC$ is a \emph{2-Calabi--Yau (CY2) category} if $\rS_{\cC} \simeq [2]$. 
\end{enumerate}
\end{definition} 

\begin{remark}
\label{remark-connected-enriques}
Let $\cC$ be a $k$-linear category, and
suppose $\tau$ is an autoequivalence of $\cC$ 
satisfying $\tau \circ \tau \simeq \id_{\cC}$. 
Then by Corollary~\ref{corollary-obstruction-HH} there is a class 
$\ob(\tau) \in \rH^3(B(\bZ/2), \HH^0(\cC/k)^{\times})$, which 
vanishes if and only if $\tau$ is the generator for a $\bZ/2$-action 
on $\cC$, in which case the set of such actions is a torsor 
under $\rH^2(B(\bZ/2), \HH^0(\cC/k)^{\times})$. 
If $\cC$ is connected, then $\HH^0(\cC/k) = k$ and 
$\rH^3(B(\bZ/2), k^{\times}) \cong \bZ/2$ and $\rH^2(B(\bZ/2), k^{\times}) = 0$ (where the $\bZ/2$-action on $k^{\times}$ is trivial), 
so a $\bZ/2$-action with $\tau$ as a generator is unique if it exists. 
This remark applies to all of the Enriques categories we consider in examples below, as they will all be connected. 
\end{remark} 

\begin{example}
\begin{enumerate}
\item If $S$ is an Enriques surface, then $\Db(S)$ is an Enriques category with $\bZ/2$-action generated by tensoring by $\omega_S$, cf.~Example~\ref{example-enriques-K3}. 
\item If $T$ is a smooth proper surface with $K_T = 0$, i.e. $T$ is a K3 or abelian surface, then $\Db(T)$ is a CY2 category. 
\end{enumerate}
\end{example} 

\begin{remark}
We use the term \emph{K3 category} to mean a CY2 category whose Hochschild homology agrees with that of the derived category of a K3 surface. 
All of the explicit examples of CY2 categories considered in this paper will in fact be K3 categories. Note that a K3 category is automatically connected, because for a CY2 category $\cC$ there is an isomorphism 
$\HH^i(\cC/k) \cong \HH_{2-i}(\cC/k)$. 
\end{remark} 

\subsection{Enriques-CY2 correspondence} 
An interesting feature of invariant categories is that they come equipped with a natural group action. In the $\bZ/2$-case, this leads to an involution on the category of $k$-linear categories equipped with a $\bZ/2$-action. 

\begin{lemma}
\label{lemma-reconstruction} 
Let $\cC$ be a $k$-linear category with a $\bZ/2$-action, 
and let $\ccD = \cC^{\bZ/2}$ be the invariant category. 
Then there is a natural $\bZ/2$-action on $\ccD$, called the \emph{residual $\bZ/2$-action}, such that there is an equivalence 
$\cC \simeq \ccD^{\bZ/2}$. 
%Moreover, the equivalence $\cC \simeq \ccD^{\bZ/2}$ is equivariant with respect to the given $\bZ/2$-action on $\cC$ and the residual $\bZ/2$-action on $\cD^{\bZ/2}$. 
\end{lemma} 

\begin{proof}
This is a special case of \cite{elagin} (where the triangulated version of the result is proved, but a similar argument also works for $k$-linear categories). 
In particular, we note that the residual action on $\ccD = \cC^{\bZ/2}$ is given by tensoring with characters of $\bZ/2$. 
\end{proof} 

Considering the invariant category of the $\bZ/2$-action on an Enriques category leads to a correspondence between Enriques and CY2 categories.  

\begin{lemma}
\label{lemma-K3-cover}  
Let $\cC$ be an Enriques category, 
and let $\ccD = \cC^{\bZ/2}$ be the invariant category for the $\bZ/2$-action. 
Then $\ccD$ is a CY2 category, called the \emph{CY2 cover} of $\cC$.
\end{lemma}

\begin{proof}
The category $\ccD$ is smooth and proper by Proposition~\ref{proposition-invariants-smooth-proper} and has Serre functor $[2]$ by \cite[Lemma 6.5]{HH}. 
\end{proof}

\begin{remark}
\label{remark-CYn} 
There is a natural generalization of Lemma~\ref{lemma-K3-cover} in which $\cC$ is assumed to be a smooth proper $k$-linear category whose Serre functor has the form $\rS_{\cC} = \sigma \circ [n]$ where $\sigma$ is the generator of a $\bZ/q$-action. 
In this case, assuming the characteristic of $k$ is coprime to $q$, we get a correspondence between such categories $\cC$ and $n$-dimensional Calabi--Yau categories $\ccD$ equipped with a residual $\bZ/q$-action. 
\end{remark} 

The source of the terminology ``CY2 cover'' in Lemma~\ref{lemma-K3-cover} 
is the following example. 

\begin{example}
\label{example-enriques-K3}
If $S$ is an Enriques surface, then its canonical bundle satisfies $\omega_S^{\otimes 2} \cong \cO_S$. 
The corresponding \'{e}tale double cover $T \to S$ is a K3 surface. 
For the $\bZ/2$-action generated by the involution of $T$ over $S$, we have 
$\Db(T)^{\bZ/2} \simeq \Db(S)$. 
Under this equivalence, the residual $\bZ/2$-action on $\Db(S)$ is generated by tensoring by $\omega_S$, and we have $\Db(S)^{\bZ/2} \simeq \Db(T)$.  
\end{example}

The following observation plays a key role in this paper, as it lets us translate the condition that two Enriques 
categories are equivalent to a statement about their CY2 covers.  

\begin{lemma}
\label{lemma-K3-cover-equivalence} 
Let $\cC_1$ and $\cC_2$ be connected Enriques categories, with CY2 covers $\ccD_1$ and $\ccD_2$. 
Then the following are equivalent: 
\begin{enumerate}
\item \label{C1-C2}
There is an equivalence $\cC_1 \simeq \cC_2$.
\item \label{C1-C2-Z2} 
There is an equivalence $\cC_1 \simeq \cC_2$ which is equivariant for the $\bZ/2$-actions generated by the $(-2)$-shifted Serre functors. 
\item \label{D1-D2-Z2} There is an equivalence $\ccD_1 \simeq \ccD_2$ which is equivariant for the residual $\bZ/2$-actions. 
\end{enumerate} 
\end{lemma}

\begin{proof}
Any equivalence $\cC_1 \simeq \cC_2$ automatically commutes with Serre functors (see e.g.~\cite[Lemma 5.4]{HH}) and shifts, and hence by connectedness and Remark~\ref{remark-connected-enriques} it is equivariant for the $\bZ/2$-actions generated by $\rS_{\cC_1}[-2]$ and $\rS_{\cC_2}[-2]$. 
This shows $\eqref{C1-C2} \iff \eqref{C1-C2-Z2}$.  
A $\bZ/2$-equivariant equivalence induces an equivalence of invariant categories, which is equivariant with respect to the residual actions. This shows $\eqref{C1-C2-Z2} \implies \eqref{D1-D2-Z2}$. 
Finally, the implication $\eqref{D1-D2-Z2} \implies \eqref{C1-C2}$ follows from Lemma~\ref{lemma-reconstruction}. 
%; hence the implications $\eqref{C1-C2-Z2} \iff \eqref{D1-D2-Z2}$ 
%follows from Lemma~\ref{lemma-reconstruction}. 
\end{proof} 

\begin{remark}
Lemma~\ref{lemma-K3-cover-equivalence} also admits an obvious generalization to the situation of Remark~\ref{remark-CYn}. 
\end{remark} 

\subsection{Examples} 
One of the main sources of Enriques and CY2 categories in this paper are Kuznetsov components of GM varieties. Generalizing the definition of $3$-dimensional GM varieties from the introduction, we have: 

\begin{definition}
\label{definition-GM} 
An $n$-dimensional \emph{GM variety}, $2 \leq n \leq 6$, is either a smooth intersection 
\begin{equation*}
W = \Gr(2,5) \cap \bP^{n+4} \cap Q
\end{equation*}
of the Pl\"{u}cker embedded Grassmannian $\Gr(2,5) \subset \bP^9$ with a linear subspace $\bP^{n+4} \subset \bP^9$ and a quadric hypersurface $Q \subset \bP^9$, or a smooth double cover 
\begin{equation*}
W \to \Gr(2,5) \cap \bP^{n+3}
\end{equation*}  
branched along $\Gr(2,5) \cap \bP^{n+3} \cap Q$ where $\bP^{n+3} \subset \bP^9$ is a linear subspace and $Q \subset \bP^9$ is a quadric hypersurface. 
We say $W$ is \emph{ordinary} in the first case, and \emph{special} in the second. 
Note that if $n = 6$ then $W$ is necessarily special.   
\end{definition} 

There is a natural correspondence between GM varieties of ordinary and special types. 

\begin{definition}
\label{definition-opposite}
The \emph{opposite} of an ordinary GM variety $W = \Gr(2,5) \cap \bP^{n+4} \cap Q$ 
of dimension $n$ is the $(n+1)$-dimensional special GM variety $W^{\op} \to \Gr(2,5) \cap \bP^{n+4}$ given by the double cover branched along $W$, 
and the \emph{opposite} of a special GM variety 
$W \to \Gr(2,5) \cap \bP^{n+3}$ of dimension $n \geq 3$ is the $(n-1)$-dimensional ordinary GM variety $W^{\op} \subset \Gr(2,5) \cap \bP^{n+3}$ given by the branch locus. 
\end{definition} 

When discussing derived categories of GM varieties, we will 
always assume for simplicity that $\characteristic(k) = 0$, 
as this is done in the references cited below; in fact, for $\characteristic(k)$ sufficiently large, all of the results still hold, but we leave the details to the interested reader. 
The Kuznetsov component of a GM variety is defined by the semiorthogonal decomposition 
\begin{equation}
\label{DbW}
\Db(W) = \llangle \Ku(W), \cU_W, \cO_W, \dots, \cU_W(\dim(W)-3), \cO_W(\dim(W)-3)\rrangle  , 
\end{equation} 
where $\cU_W$ and $\cO_W(1)$ denote the pullbacks to $W$ of the tautological rank $2$ subbundle and Pl\"{u}cker line bundle on $\Gr(2,5)$.  
These categories were extensively studied in \cite{GM-derived}. 
Note that if $W$ is a GM threefold, this agrees with the definition from~\eqref{KuX}. 
Further, if $W$ is a GM surface, then $W$ is a K3 surface of degree $10$ and $\Ku(W) = \Db(W)$. 

\begin{remark}
Instead of the Kuznetsov component as we have defined it, 
\cite{GM-derived} studies a category $\cA_W$ defined by the slightly different semiorthogonal decomposition 
\begin{equation*}
\Db(W) = \llangle \cA_W, \cO_W, \cU^{\svee}_W, \dots, \cO_W(\dim(W) - 3), \cU^{\svee}_W(\dim(W) - 3) \rrangle. 
\end{equation*} 
There is a canonical equivalence $\Ku(W) \simeq \cA_W$ given by $\Phi = \rL_{\cO_W} \circ (- \otimes \cO_W(1))$, where $\rL_{\cO_W}$ is the left mutation functor through $\cO_W$. 
Indeed, if we tensor the defining semiorthogonal decomposition~\eqref{DbW} of $\Ku(W)$ by $\cO_W(1)$ and mutate the object $\cO_W(\dim(W)-2)$ to the far left, we get 
\begin{equation*}
\Db(W) = \llangle \cO_W, \Ku(W) \otimes \cO_W(1), \cU_W(1), \cO_W(1), \dots, \cU_W(\dim(W)-2) \rrangle. 
\end{equation*} 
Using that $\cU_W(1) \cong \cU_W^{\svee}$ and mutating $\Ku(W) \otimes \cO_W(1)$ through $\cO_W$ we obtain 
\begin{equation*}
\Db(W) = \llangle \Phi( \Ku(W) ), \cO_W, \cU_W^{\svee}, \dots, \cO_W(\dim(W) - 3), \cU_W^{\svee}(\dim(W)-3) \rrangle, 
\end{equation*} 
from which the stated equivalence follows, cf.~\cite[Lemma 2.30]{GM-derived}. 
We chose to define $\Ku(W)$ by the decomposition~\eqref{DbW} for consistency with the definition~\eqref{KuX} from \cite{kuznetsov-fano3fold} in the case of a GM threefold. 
\end{remark} 

Recall from~\eqref{KuY} the definition of the Kuznetsov component of a quartic double solid. 

\begin{proposition}{\label{proposition-enriques-examples}}
\begin{enumerate}
\item \label{quartic-double-enriques} 
If $Y$ is a quartic double solid, then $\Ku(Y)$ is a 
connected Enriques category. 
\item \label{GM-enriques} 
If $\characteristic(k) = 0$ and $X$ is an odd-dimensional GM variety, then $\Ku(X)$ is a 
connected Enriques category. 
\item \label{GM-K3} If $\characteristic(k) = 0$ and $W$ is an even-dimensional GM variety, then $\Ku(W)$ is a 
connected K3 category. 
\end{enumerate}
\end{proposition} 

\begin{proof}
By \cite{kuznetsov-CY} and the Hochschild homology computation in \cite[Proposition 2.9]{GM-derived}, the category $\Ku(W)$ in~\eqref{GM-K3} is a K3 category, while the Serre functors in~\eqref{quartic-double-enriques} and~\eqref{GM-enriques} are of the form 
$\tau \circ [2]$ for an involution~$\tau$. 
By \cite[Proposition 2.6]{GM-derived} the involution $\tau$ is nontrivial for $\Ku(X)$, and the same argument applies to $\Ku(Y)$. 
Thus to show that $\Ku(Y)$ and $\Ku(X)$ are Enriques categories, it remains to show that there is a $\bZ/2$-action with generator~$\tau$. 
By Remark~\ref{remark-connected-enriques} there is a potential obstruction to the existence of such an action; we show it vanishes by relating $\tau$ to a geometric $\bZ/2$-action. 

More precisely, in case~\eqref{quartic-double-enriques} or in case~\eqref{GM-enriques} if $X$ is special, \cite{kuznetsov-CY} shows that $\tau$ can be described as the pushforward along the covering involution of $Y$ or $X$; in particular, it follows that $\tau$ is the generator of a $\bZ/2$-action. 
If $X$ is ordinary, then $\tau \simeq \Phi_X[-1]$ where $\Phi_X$ is the ``rotation functor'' defined in~\eqref{rotation-functor} below. 
Note that the Kuznetsov component $\Ku(X^{\op})$ of the opposite variety has a $\bZ/2$-action generated by the covering involution. 
By \cite[\S8.2]{cyclic-covers} there is an equivalence $\Ku(X^{\op})^{\bZ/2} \simeq \Ku(X)$, such that $\Phi_X[-1]$ is the generator of the residual $\bZ/2$-action on $\Ku(X)$; in particular, $\tau$ also corresponds to a $\bZ/2$-action in this case. 

Finally, the connectedness of the categories in~\eqref{GM-enriques} and~\eqref{GM-K3} holds by the computation of their Hochschild cohomology in \cite[Corollary 2.11 and Proposition 2.12]{GM-derived}, and an analogous argument applies to the categories in~\eqref{quartic-double-enriques}.  
\end{proof} 

The CY2 covers of Kuznetsov components of quartic double solids and 
odd-dimensional GM varieties can be described explicitly as follows. 

\begin{theorem}[{\cite{cyclic-covers}}]
\label{theorem-K3-cover-examples} 
\begin{enumerate}
\item \label{quartic-involution}
Let $Y \to \bP^3$ be a quartic double solid with branch locus $Y_{\br}$. 
Then the CY2 cover of $\Ku(Y)$ is equivalent to $\Db(Y_{\br})$. 
Under this equivalence the residual $\bZ/2$-action on $\Db(Y_{\br})$ is generated by the  autoequivalence $\Phi_{Y_{\br}}^2 [-1]$, where 
\begin{equation*}
\Phi_{Y_{\br}} = \rT_{\cO_{Y_{\br}}} \circ (- \otimes \cO_{Y_{\br}}(1))
\end{equation*} 
and $\rT_{\cO_{Y_{\br}}}$ is the spherical twist around $\cO_{Y_{\br}}$. 

\item \label{GM-involution}
Let $W$ be a GM variety and assume $\characteristic(k) = 0$. 
Let 
\begin{equation}
\label{rotation-functor} 
\Phi_{W}  = \rL_{\cU_{W}} \circ \rL_{\cO_{W}}  \circ (- \otimes \cO_{W}(1)) 
\colon \Ku(W) \to \Ku(W)  
\end{equation} 
where $\rL_{\cU_{W}}$ and $\rL_{\cO_{W}}$ denote the left mutation functors 
through $\cU_{W}$ and $\cO_{W}$ if $\dim(W) \geq 3$ 
(in which case these objects are exceptional), or the spherical spherical twists 
around $\cU_{W}$ and $\cO_{W}$ if $\dim(W) = 2$ (in which case 
these objects are spherical). 
Then $\Phi_{W}[-1]$ is an involutive autoequivalence of $\Ku(W)$, such that: 
\begin{itemize}
\item If $W$ is special then $\Phi_{W}[-1] \simeq i_*$ where $i$ is 
the covering involution of $W$. 
\item If $\dim W$ is odd then 
$\Phi_{W}[-1] \simeq \tau$ where $\tau = \rS_{\Ku(W)} [-2]$. 
\end{itemize} 

\item \label{K3-cover-GM} 
Let $W$ be an odd-dimensional GM variety and assume $\characteristic(k) = 0$.
Then the CY2 cover of $\Ku(W)$ is equivalent to $\Ku(W^{\op})$. 
Under this equivalence the residual $\bZ/2$-action on $\Ku(W^{\op})$ is generated by
the autoequivalence $\Phi_{W^{\op}} [-1]$
\end{enumerate} 
\end{theorem}

\begin{proof}
This follows from the main results of \cite{cyclic-covers}, as explained in  \cite[\S8.1-8.2]{cyclic-covers}. 
\end{proof}

\begin{remark}
In the upcoming paper \cite{moduli-enriques}, we will use the description of the CY2 covers in these and other examples to describe moduli spaces of Bridgeland stable objects in Enriques categories.  
\end{remark}

%%%%%%%%%%%%%%%%%%%%%%%%%%%%%%%%%%%%%%%%%%%%%%%%%%%%%%

\section{Hodge theory via Kuznetsov components} 
\label{section-Hodge-theory}

In this section we work over the complex numbers. 
After reviewing some facts about Mukai Hodge structures in \S\ref{mukai-HS}, 
we prove Proposition~\ref{proposition-GM-periods} on the categorical description of GM periods in \S\ref{subsection-proof-GM-periods}, and give an application to the period map in \S\ref{subsection-application-periods}. 

\subsection{Mukai Hodge structure} 
\label{mukai-HS}
As explained in \cite{IHC-CY2}, to any $\bC$-linear category $\cC$ occuring as a semiorthogonal component in the derived category of a smooth proper complex variety, one can canonically attach a lattice equipped with a Hodge structure, which is additive under semiorthogonal decompositions. 
If $\cC$ is a CY2 category then we get a 
weight $2$ Hodge structure $\tH(\cC, \bZ)$, called the \emph{Mukai Hodge structure} of $\cC$ (see \cite[Definition~6.4]{IHC-CY2}), which generalizes a construction of Addington--Thomas \cite{addington-thomas} in the case of Kuznetsov components of cubic fourfolds. 
Below we explicitly describe this Hodge structure in the cases of interest 
for this paper. 

\begin{example}
\label{example-tH-K3}
Let $T$ be a complex K3 surface. 
Then $\tH(\Db(T), \bZ)$ coincides with the classical Mukai Hodge structure, defined as 
\begin{equation*}
\tH(T, \bZ) = \rH^0(T, \bZ)(-1) \oplus \rH^2(T, \bZ) \oplus \rH^4(T, \bZ)(1) 
\end{equation*} 
where $(-1)$ and $(1)$ denote Tate twists, 
with pairing $((r,c,s) , (r',c',s') ) = c  c'-r  s'-r'  s$. 
\end{example} 

\begin{example}
Let $W$ be a GM fourfold. In this case, $\tH(\Ku(W), \bZ)$ was originally defined and studied in \cite{pertusi-GM}, and admits the following explicit description.  
As an abelian group, 
\begin{equation*}
\tH(\Ku(W), \bZ) = \set{ v \in \Ktop[0](W) \st \chi_{\rtop}([\cU_W(i)], v) = \chi_{\rtop}([\cO_W(i)], v) = 0 \text{ for } i = 0,1}
\end{equation*} 
where $\Ktop[0](X)$ denotes the complex topological $K$-theory of the space of complex points $X(\bC)$, and $\chi_{\rtop}$ denotes the topological Euler pairing. 
The group $\tH(\Ku(W), \bZ)$ is regarded as a lattice with the symmetric pairing $(-,-) = - \chi_{\rtop}(-,-)$. 
Further, the Chern character induces an embedding  
\begin{equation*}
\tH(\Ku(W), \bZ) \otimes \bQ \to \rH^{\mathrm{even}}(W, \bQ). 
\end{equation*} 
By taking appropriate Tate twists, we regard $ \rH^{\mathrm{even}}(W, \bQ) = \bigoplus_{k = 0}^4 \rH^{2k}(X, \bQ)(k-1)$ as a weight $2$ Hodge structure. 
The Hodge filtration on $\tH(\Ku(W), \bZ) \otimes \bC$ is then the intersection of the corresponding filtration on $
\rH^{\mathrm{even}}(W, \bC)$ along the above embedding. 
\end{example} 

\begin{example}
Let $W$ be a GM sixfold. 
Then there is a similar description of the Mukai Hodge structure: as an abelian group 
\begin{equation*}
\tH(\Ku(W), \bZ) = \set{ v \in \Ktop[0](W) \st \chi_{\rtop}([\cU_W(i)], v) = \chi_{\rtop}([\cO_W(i)], v) = 0 \text{ for } i = 0,1,2,3} , 
\end{equation*} 
with Hodge structure induced by the one on $\rH^{\mathrm{even}}(W, \bQ)$. 
\end{example} 

The above description shows that for a GM fourfold or sixfold, 
$\tH(\Ku(W), \bZ)$ is rationally quite closely related to the usual middle-degree cohomology $\rH^{\dim(W)}(W, \bZ)$. 
We will need the following integral relation. 

\begin{proposition}[\cite{pertusi-GM}]
\label{proposition-primitive-cohomology} 
Let $W$ be a GM variety of dimension $n=4$ or $6$. 
\begin{enumerate}
\item \label{A12} There is a canonical rank $2$ sublattice
\begin{equation*}
A_1^{\oplus 2} = \begin{pmatrix}
2 & 0 \\ 
0 & 2 
\end{pmatrix}   \subset \tH(\Ku(W), \bZ)
\end{equation*} 
which is the image of the map $\rK_0(\Gr(2,5)) = \Ktop[0](\Gr(2,5)) \to \tH(\Ku(W), \bZ)$ given by 
pulling back classes on $\Gr(2,5)$ and projecting into $\Ku(W)$. 

\item \label{Hprim} 
Let $\tH(\Ku(W), \bZ)_0$ denote the orthogonal sublattice to $A_1^{\oplus 2} \subset 
\tH(\Ku(W), \bZ)$, and let $\rH^n(W, \bZ)_0$ denote the orthogonal sublattice to  
$\rH^n(\Gr(2,5), \bZ) \hookrightarrow \rH^n(W, \bZ)$. Then the Chern class 
$c_{n/2} \colon \Ktop[0](W) \to \rH^n(W, \bZ)$ induces an isometry of 
weight $2$ Hodge structures 
\begin{equation*}
\tH(\Ku(W), \bZ)_0 \cong \rH^n(W, \bZ)_0(\tfrac{n}{2}-1), 
\end{equation*} 
where $(\frac{n}{2}-1)$ on the right denotes a Tate twist. 
\end{enumerate}
\end{proposition} 

\begin{remark}
Proposition~\ref{proposition-tH-A12} below 
implies that our notation $\tH(\Ku(W), \bZ)_0$ above is consistent with that of Proposition~\ref{proposition-GM-periods}. 
\end{remark} 

\begin{proof}
Part~\eqref{A12} follows from \cite[Lemma 2.27]{GM-derived}. 
Part~\eqref{Hprim} for $n = 4$ holds by \cite[Proposition~3.1]{pertusi-GM}, 
and the $n = 6$ case holds by an analogous argument. 
\end{proof}

For later use in \S\ref{section-nonexistence}, we record a lift of the isomorphism in Proposition~\ref{proposition-primitive-cohomology}\eqref{Hprim} to the level of quotient groups. 
We only consider the $4$-dimensional case, as that is the one we shall need, but a 
similar statement holds in dimension $6$. 
\begin{proposition}[\cite{pertusi-GM}]
\label{proposition-tH}
Let $W$ be a GM fourfold. 
There is an isomorphism of abelian groups 
\begin{equation*}
\frac{\tH(\Ku(W), \bZ)}{A_1^{\oplus 2}} \cong \frac{\rH^4(W, \bZ)}{\rH^4(\Gr(2,5), \bZ)} 
\end{equation*} 
induced by the second Chern class $c_2 \colon \tH(\Ku(W), \bZ) \to \rH^4(W, \bZ)$. 
Moreover, under the resulting correspondence between sublattices 
\begin{equation*}
\rH^4(\Gr(2,5), \bZ) \subset K \subset \rH^4(W, \bZ) \quad \text{and} \quad 
A_1^{\oplus 2} \subset K' \subset \tH(\Ku(W), \bZ) 
\end{equation*} 
we have: 
\begin{enumerate}
\item \label{K-primitive} $K$ is primitive if and only if $K'$ is primitive. 
\item \label{K-disc} $K$ is non-degenerate if and only if $K'$ is nondegenerate, 
in which case  
$K$ has signature $(r,s)$ if and only if $K'$ has signature $(s+2, r-2)$ 
and $\disc(K) = (-1)^{\rk K}\disc(K')$. 
\item \label{K-hodge} $K \subset \rH^{4}(W, \bZ)$ consists of Hodge classes if and only if $K' \subset \tH(\Ku(W), \bZ)$ consists of Hodge classes. 
\end{enumerate}
\end{proposition}

\begin{proof} 
The isomorphism is \cite[Propositions 3.2]{pertusi-GM}, 
\eqref{K-primitive} and~\eqref{K-disc} hold by \cite[Lemma 3.4]{pertusi-GM} (taking into account that our lattice $\tH(\Ku(W), \bZ)$ is by definition the negative of the one considered there), and 
\eqref{K-hodge} 
follows from Proposition~\ref{proposition-primitive-cohomology}\eqref{Hprim}. 
\end{proof} 

\subsection{Proof of Proposition~\ref{proposition-GM-periods}} 
\label{subsection-proof-GM-periods}
Proposition~\ref{proposition-primitive-cohomology} reduces 
Proposition~\ref{proposition-GM-periods} to the following result.  
Note that the construction of Hodge structures for categories from \cite{IHC-CY2} 
is functorial; in particular, a $\bZ/2$-action on a CY2 category does indeed 
induce a $\bZ/2$-action on its Mukai Hodge structure.

\begin{proposition}
\label{proposition-tH-A12}
Let $W$ be a GM variety of dimension $4$ or $6$. 
Then the invariant sublattice $\tH(\Ku(W), \bZ)^{\bZ/2} \subset  \tH(\Ku(W), \bZ)$ for the residual $\bZ/2$-action equals the canonical sublattice $A_1^{\oplus 2} \subset \tH(\Ku(W), \bZ)$. 
\end{proposition} 

\begin{remark} 
The residual $\bZ/2$-action on $\tH(\Ku(W), \bZ)$ is by isometries, so its $-1$-eigenspace must be exactly the orthogonal complement $\tH(\Ku(W), \bZ)_0$ of $A_1^{\oplus 2}$. 
%Since the residual $\bZ/2$-action induces isometries on cohomology, its $-1$-eigenspace must be exactly the orthogonal complement $\tH(\Ku(W), \bZ)_0$ of $A_1^{\oplus 2}$.
In particular, the residual $\bZ/2$-action on $\Ku(W)$ is \emph{antisymplectic}, 
in the sense that it acts by $-1$ on $\tH^{2,0}(\Ku(W))$.  
This induces antisymplectic involutions of Bridgeland moduli spaces of objects 
in $\Ku(W)$ with class in $A_1^{\oplus 2}$, 
giving a categorical interpretation for the antisymplectic involutions from \cite[Proposition 5.16]{GM-stability} whose 
existence was guaranteed there lattice theoretically. 
The geometry of the fixed loci of these involutions will be described in the upcoming work \cite{moduli-enriques}. 
\end{remark} 

We will prove Proposition~\ref{proposition-tH-A12} by showing the claim is 
deformation invariant, and then checking it for a specific $W$ where the claim 
is easy. 
For this, we will need to consider families of GM varieties and their 
Kuznetsov components. 
By a family of GM varieties over a base $S$, we mean a smooth proper morphism 
$\pi \colon \cW \to S$ equipped with a line bundle $\cO_{\cW}(1)$ on $\cW$, such that for every point $s \in S$ the pair $(\cW_s, \cO_{\cW_s}(1))$ is a GM variety with the Pl\"{u}cker polarization. 

The results of \cite{DebKuz:birGM} show that for any family of GM varieties, there is a canonical rank $5$ vector bundle $\cV_5$ on $S$ and a morphism $\cW \to \Gr_S(2, \cV_5)$ which on fibers recovers the usual map to $\Gr(2,5)$. 
We denote by $\cU_{\cW}$ the pullback to $\cW$ of the tautological rank $2$ 
subbundle on $\Gr_S(2,\cV_5)$. 

\begin{lemma}
\label{lemma-relative-KuGM}
Let $\pi \colon \cW \to S$ be a family of $n$-dimensional GM varieties. 
\begin{enumerate}
\item \label{relative-KuGM}
There is an $S$-linear semiorthogonal decomposition 
\begin{multline*}
\Dperf(\cW)  = \langle \Ku(\cW),  \pi^*(\Dperf(S)) \otimes \cU_{\cW}, \pi^*(\Dperf(S)) \otimes \cO_{\cW}, 
\dots \\ 
\dots, \pi^*(\Dperf(S)) \otimes \cU_{\cW}(n-3), \pi^*(\Dperf(S)) \otimes \cO_{\cW}(n-3) \rangle 
\end{multline*} 
such that the fiber of $\Ku(\cW)$ over any point $s \in S$ satisfies $\Ku(\cW)_s \simeq \Ku(\cW_s)$, 
where the right side is defined by~\eqref{DbW}. 
\item \label{relative-OW}
Let 
\begin{equation*}
\Phi_{\cW} = \rL_{\cU_{\cW}/S} \circ \rL_{\cO_{\cW}/S}  \circ (- \otimes \cO_{\cW}(1)) 
\colon \Ku(\cW) \to \Ku(\cW) 
\end{equation*} 
where for $E = \cU_{W}$ or $\cO_{\cW}$ the functor $\rL_{E/S}$ is defined by 
the exact triangle 
\begin{equation*}
\pi^* \cHom_S(E, F) \otimes E \to F \to \rL_{E/S}(F) 
\end{equation*} 
for $F \in \Dperf(\cW)$. 
Then $\Phi_{\cW}$ is an autoequivalence of $\Ku(\cW)$, whose fiber over 
any $s \in S$ recovers the autoequivalence $\Phi_{\cW_s}$ from Theorem~\ref{theorem-K3-cover-examples}\eqref{GM-involution} under 
the identification $\Ku(\cW)_s \simeq \Ku(\cW_s)$. 
\end{enumerate}
\end{lemma} 

\begin{remark}
\label{remark-involution-fibers}
Note that we do not claim $\Phi_{\cW}[-1]$ is necessarily an involution, but it is fiberwise an involution by 
Theorem~\ref{theorem-K3-cover-examples}\eqref{GM-involution}. 
\end{remark} 

\begin{proof} 
The semiorthogonal decomposition in~\eqref{relative-KuGM} follows from the 
decomposition \eqref{DbW} on fibers, cf.~\cite[Lemma 3.22 and 3.25]{stability-families} and \cite{samokhin}, and the statement about the fibers of $\Ku(\cW)$ is a consequence of the compatibility of base change with semiorthogonal decompositions. 
For $E = \cU_W$ or $\cO_W$, the functor 
$\rL_{E/S}$ is nothing but the left mutation through the admissible subcategory $\pi^*(\Dperf(S)) \otimes E$ if $n \geq 3$, and the spherical twist around the spherical functor 
$\Dperf(S) \to \Dperf(\cW), F \mapsto \pi^*(F) \otimes E$ if $n = 2$; 
it follows in particular that $\Phi_{\cW}$ is indeed an autoequivalence of $\Ku(\cW)$. 
The final claim about the fibers of $\Phi_{\cW}$ is immediate from the definitions. 
\end{proof}

\begin{proof}[Proof of Proposition~\ref{proposition-tH-A12}]
Suppose $\pi \colon \cW \to S$ is a family of GM varieties of dimension $n = 4$ or $6$. 
By \cite{IHC-CY2} the Mukai Hodge structures $\tH(\Ku(\cW_s), \bZ)$, $s \in S(\bC)$, form the 
fibers of a local system $\tH(\Ku(\cW)/S, \bZ)$ on the analytification $S^{\an}$. 
By functoriality the autoequivalence $\Phi_{\cW}[-1]$ induces an automorphism 
of the local system $\tH(\Ku(\cW)/S, \bZ)$, which is necessarily an involution because it is so on fibers (see Remark~\ref{remark-involution-fibers}). 
Therefore, we have a $\bZ/2$-action on the local system 
$\tH(\Ku(\cW)/S, \bZ)$ which fiberwise recovers the residual 
$\bZ/2$-action on the Mukai Hodge structures. 
So it suffices to prove the proposition for any particular fiber of the family $\cW \to S$. 

We may thus assume $W$ is a special GM variety. 
Then by Theorem~\ref{theorem-K3-cover-examples} 
the residual $\bZ/2$-action on $\Ku(W)$ is induced by 
the covering involution of $W$. 
By construction $\tH(\Ku(W), \bZ)$ is a summand of $\Ktop[0](W)$, 
and by the previous remark the inclusion $\tH(\Ku(W), \bZ) \subset \Ktop[0](W)$ is $\bZ/2$-equivariant, where $\bZ/2$ acts by the residual action on the left and the covering involution on the right. 
By the description of $A_1^{\oplus 2} \subset  \tH(\Ku(W), \bZ)$ from Proposition~\ref{proposition-primitive-cohomology}\eqref{A12}, 
it follows that $A_1^{\oplus 2} \subset \tH(\Ku(W), \bZ)^{\bZ/2}$. 
Moreover, as $\rH^4(\Gr(2,5), \bZ) \subset \rH^4(W, \bZ)$ is primitive by \cite[\S5.1]{DIM4fold}, it follows from Proposition~\ref{proposition-tH}\eqref{K-primitive} that 
$A_1^{\oplus 2} \subset \tH(\Ku(W), \bZ)$ is also primitive. 
Therefore, to finish the proof it suffices to show the inclusion $A_1^{\oplus 2} \subset \tH(\Ku(W), \bZ)^{\bZ/2}$ is rationally an isomorphism. 
By applying the Chern character, 
it is enough to show $\rH^*(\Gr(2,5), \bQ)$ surjects onto the invariants of 
$\rH^*(W, \bQ)$. 
If $\dim(W) = 6$ this is true because $W \to \Gr(2,5)$ is a double cover. 
If $\dim(W) = 4$ then $W \to \Gr(2,5) \cap \bP^8$ is a double cover, so it is enough to observe that 
$\rH^*(\Gr(2,5), \bQ)$ surjects onto $\rH^*(\Gr(2,5) \cap \bP^8, \bQ)$; this follows, for instance, from the semiorthogonal decompositions 
\begin{align*}
\Db(\Gr(2,5)) & =  \llangle \cO, \cU^{\svee}, \dots, \cO(4), \cU^{\svee}(4) \rrangle , \\  
\Db(\Gr(2,5) \cap \bP^8 ) & = \llangle \cO, \cU^{\svee}, \dots, \cO(3), \cU^{\svee}(3) \rrangle , 
\end{align*}
which hold by \cite[\S6.1]{kuznetsov-hyperplane}. 
\end{proof} 

\subsection{Application to periods} 
\label{subsection-application-periods}
In \cite{DebKuz:birGM} Debarre and Kuznetsov classified 
GM varieties in terms of linear algebraic data, 
by constructing for any GM variety $W$ a \emph{Lagrangian data set} $(V_6(W), V_5(W), A(W))$, where 
\begin{itemize} 
\item $V_6(W)$ is a $6$-dimensional vector space, 
\item $V_5(W) \subset V_6(W)$ is a hyperplane, and 
\item $A(W) \subset \wedge^3 V_6(W)$ is a Lagrangian subspace with respect to the wedge product, 
\end{itemize} 
and proving that $W$ is completely determined by its dimension and this data.  
Surprisingly, many properties of $W$ only depend on $A(W)$. 
Recall that GM varieties $W_1$ and $W_2$ with $\dim(W_1) \equiv \dim(W_2) \pmod 2$ 
are called \emph{generalized partners} if there exists an isomorphism $V_6(W_1) \cong V_6(W_2)$ identifying $A(W_1) \subset \wedge^3 V_6(W_1)$ with $A(W_2) \subset \wedge^3V_6(W_2)$, and  
\emph{generalized duals} if there exists an isomorphism $V_6(W_1) \cong V_6(W_2)^{\svee}$ identifying $A(W_1) \subset \wedge^3 V_6(W_1)$ with $A(W_2)^{\perp} \subset \wedge^3V_6(W_2)^{\svee}$. 
In the case where $\dim(W_1) = \dim(W_2)$, these definitions specialize to the notions of \emph{period partners} and \emph{duals} originally introduced in \cite{DebKuz:birGM}. 

The main result of \cite{DebKuz:periodGM} shows that $A(W)$ determines the period of $W$ in dimensions $4$ and $6$, 
which implies the period map factors through the moduli space of EPW sextics studied by O'Grady \cite{OG-EPW}. More precisely: 

\begin{theorem}[\cite{DebKuz:periodGM}]
\label{theorem-periods-DK}
Let $W_1$ and $W_2$ be GM varieties of (possibly unequal) dimensions $n_1, n_2 \in \{4,6\}$. 
Assume $W_1$ and $W_2$ are generalized partners or duals. Then there is an isometry of Hodge structures 
\begin{equation*}
\rH^{n_1}(W_1, \bZ)_0(\tfrac{n_1}{2}) \cong \rH^{n_2}(W_2, \bZ)_0(\tfrac{n_2}{2}) . 
\end{equation*} 
\end{theorem} 

This is proved in \cite{DebKuz:periodGM} by intricate geometric arguments, 
but here we note a simple categorical proof. 
This relies on the following slight enhancement of \cite[Theorem 1.6]{categorical-cones}. 

\begin{theorem}
\label{theorem-duality} 
Let $W_1$ and $W_2$ be GM varieties which are not special GM surfaces. 
If $W_1$ and $W_2$ are generalized partners or duals, then there is an 
equivalence $\Ku(W_1) \simeq \Ku(W_2)$ which is equivariant for the 
canonical $\bZ/2$-actions described in Theorem~\ref{theorem-K3-cover-examples}\eqref{GM-involution}. 
\end{theorem} 

\begin{proof}
By \cite[Theorem 1.6]{categorical-cones} there exists an equivalence 
$\Ku(W_1) \simeq \Ku(W_2)$, not a priori $\bZ/2$-equivariant. 
If $W_1$ and $W_2$ are odd-dimensional, then by Lemma~\ref{lemma-K3-cover-equivalence} this equivalence is necessarily $\bZ/2$-equivariant. 
If $W_1$ and $W_2$ are even-dimensional, then by Theorem~\ref{theorem-K3-cover-examples}\eqref{K3-cover-GM} combined with Lemma~\ref{lemma-K3-cover-equivalence}, 
it suffices to show there is an equivalence $\Ku(W_1^{\op}) \simeq \Ku(W_2^{\op})$. 
But by construction the Lagrangian data of a GM variety and its opposite GM variety are the same, so $W_1^{\op}$ and $W_2^{\op}$ are odd-dimensional generalized partners or duals, and the equivalence holds again by \cite[Theorem 1.6]{categorical-cones}. 
\end{proof} 

\begin{proof}[Proof of Theorem~\ref{theorem-periods-DK}]
The $\bZ/2$-equivariant equivalence $\Ku(W_1) \simeq \Ku(W_2)$ of Theorem~\ref{theorem-duality} 
induces a $\bZ/2$-equivariant isomorphism $\tH(\Ku(W_1), \bZ) \cong \tH(\Ku(W_2), \bZ)$, 
which by Proposition~\ref{proposition-tH-A12} must identify the canonical copy of $A_1^{\oplus 2}$ on each side. 
Now the result follows from Proposition~\ref{proposition-primitive-cohomology}\eqref{Hprim}. 
\end{proof}

%%%%%%%%%%%%%%%%%%%%%%%%%%%%%%%%%%%%%%%%%%%%%%%%%%%%%%

\section{Birational categorical Torelli} 
\label{section-torelli}

In this section we work over the complex numbers. 
After reviewing some facts about the period map for GM fourfolds, 
we prove Theorem~\ref{main-theorem-3} in \S\ref{section-proof-main-theorem-3}. 

\subsection{The period morphism}
\label{section-period-morphism}
Let $\cN$ denote the moduli stack of GM fourfolds. 
This is a smooth irreducible $24$-dimensional Deligne--Mumford stack of finite type over $\bC$ (see \cite[Proposition A.2]{GM-derived} or \cite{DebKuz:moduli}). 
We denote by 
\begin{equation*}
p \colon \cN \to \cD 
\end{equation*}
the period map, where the period domain $\cD$ is the $20$-dimensional quasi-projective variety classifying Hodge structures on the middle cohomology $\rH^4(W_0, \bZ)$ of a fixed GM fourfold $W_0$ for which the canonical rank $2$ sublattice $\rH^4(\Gr(2,5),\bZ) \subset \rH^4(W_0, \bZ)$ consists of Hodge classes (see \cite{DIM4fold} for details). 
We note that $\cD$ is equipped with a canonical involution, denoted $r_{\cD}$ (see the discussion preceding \cite[Corollary 6.3]{DIM4fold}). 

We will also need to consider the related moduli space
$\cN^{\EPW}$ of EPW sextics \cite{OGrady-moduli}. 
Recall that, if we fix $V_6$ a $6$-dimensional vector space, then 
$\cN^{\EPW}$ is the GIT quotient by $\PGL(V_6)$ of the space of Lagrangians $A \subset \wedge^3 V_6$ 
containing no decomposable vectors. The space $\cN^{\EPW}$ has a natural involution $r^{\EPW}$, induced by sending $A \subset \wedge^3 V_6$ to its orthogonal $A^{\perp} \subset \wedge^4 V_6^{\svee}$. 
O'Grady constructed a period morphism 
\begin{equation*}
p^{\EPW} \colon \cN^{\EPW} \to \cD,
\end{equation*}
sending $A$ to the period of the associated double EPW sextic \cite{ogrady2015periods}. 

The results of \cite{DebKuz:birGM, DebKuz:moduli} show that there is a surjective morphism 
\begin{equation*}
\pi \colon \cN \to \cN^{\EPW},
\end{equation*} 
sending a GM fourfold $W$ to its Lagrangian $A(W)$ (as in \S\ref{subsection-application-periods}). 
In particular, by definition GM fourfolds $W_1$ and $W_2$ are period partners if and only if 
$\pi(W_1) = \pi(W_2)$, and duals if and only if $\pi(W_1) = r^{\EPW}(\pi(W_2))$. 

We will need the following two ingredients in our proof of Theorem~\ref{main-theorem-3}. 

\begin{theorem}[{\cite{DebKuz:periodGM}}]
\label{theorem-period-factorization}
There is a factorization  $p = p^{\EPW} \circ \pi$. 
\end{theorem} 

\begin{proof}
This follows from \cite{DebKuz:periodGM}, cf. \cite[Proposition 6.1]{DebKuz:moduli} and Theorem~\ref{theorem-periods-DK} above. 
\end{proof} 

\begin{theorem}
\label{theorem-EPW-torelli}
The morphism $p^{\EPW} \colon \cN^{\EPW} \to \cD$ is an open embedding and commutes with the natural involutions, i.e. $p^{\EPW} \circ r^{\EPW} = r_{\cD} \circ p^{\EPW}$. 
\end{theorem} 

\begin{proof}
That $p^{\EPW}$ is an open embedding follows from Verbitsky's Torelli theorem; see \cite[Theorem 1.3]{ogrady2015periods} for a more precise statement. 
That $p^{\EPW}$ commutes with the involutions is proved in \cite{OG:dualEPW}. 
\end{proof} 

\subsection{Proof of Theorem~\ref{main-theorem-3}} 
\label{section-proof-main-theorem-3}
By the discussion preceding the statement of Theorem~\ref{main-theorem-3}, 
we only need to prove the forward implication. 
So let $X_1$ and $X_2$ be GM threefolds such that $\Ku(X_1) \simeq \Ku(X_2)$. 
Note that if $X$ is a special GM threefold, then we may find an ordinary GM threefold $X'$ which is a period partner of $X$ (see e.g. \cite[Lemma 3.8]{GM-derived}), and hence satisfies   
$\Ku(X) \simeq \Ku(X')$ by \cite[Theorem 1.6]{categorical-cones}. 
Therefore, we may assume that $X_1$ and $X_2$ are both ordinary. 
By Proposition~\ref{proposition-enriques-examples}, Lemma~\ref{lemma-K3-cover-equivalence}, and Theorem~\ref{theorem-K3-cover-examples}, passing to CY2 covers gives an equivalence 
$\Ku(X_1^{\op}) \simeq \Ku(X_2^{\op})$ which is equivariant for the residual $\bZ/2$-actions. 
This induces a $\bZ/2$-equivariant Hodge isometry  
\begin{equation*}
\tH(\Ku(X_1^{\op}), \bZ) \xrightarrow{\, \sim \,} \tH(\Ku(X_2^{\op}), \bZ),  
\end{equation*}
and hence by Proposition~\ref{proposition-GM-periods} a Hodge isometry 
$\rH^4(X_1^{\op}, \bZ)_0 \cong \rH^4(X_2^{\op}, \bZ)_0$. 
By the definition of the period morphism, it follows that either $p(X_1^{\op}) = p(X_2^{\op})$ or 
$p(X_1^{\op}) = r_{\cD}(p(X_2^{\op}))$, cf. \cite[Lemma 5.26]{DebKuz:periodGM}. 
By Theorems~\ref{theorem-period-factorization} and~\ref{theorem-EPW-torelli}, this 
means that either $\pi(X_1^{\op}) = \pi(X_2^{\op})$ or $\pi(X_1^{\op}) = r^{\EPW}(\pi(X_2^{\op}))$, i.e. 
either $X_1^{\op}$ and $X_2^{\op}$ are period partners or duals. 
As these relations are preserved under passing to opposite GM varieties, we conclude the same is true of $X_1$ and $X_2$. \qed

%%%%%%%%%%%%%%%%%%%%%%%%%%%%%%%%%%%%%%%%%%%%%%%%%%%%%%

\section{Nonexistence of equivalences} 
\label{section-nonexistence}

In this section we work over the complex numbers. 
After reviewing some restrictions on the periods of GM fourfolds and surfaces, 
we prove Theorem~\ref{main-theorem} in \S\ref{section-proof-main-theorem}. 

\subsection{Restrictions on periods} 
As in our discussion of the period map $p \colon \cN \to \cD$ for GM fourfolds in \S\ref{section-period-morphism}, let $W_0$ be a fixed GM fourfold. 
Let $K \subset \rH^4(W_0, \bZ)$ be a rank $3$ primitive positive definite sublattice which contains $\rH^4(\Gr(2,5), \bZ)$. 
We consider the locus in $\cD$ parameterizing Hodge structures on 
$\rH^4(W_0, \bZ)$ for which $K \subset \rH^4(W_0, \bZ)$ consists of 
Hodge classes. 
By \cite{DIM4fold} this locus only depends on the discriminant $d>0$ of $K$, 
so we denote it by $\cD_d$. 
Moreover, the locus $\cD_d$ is nonempty for $d \equiv 0, 2, 4 \pmod 8$, 
an irreducible divisor for $d \equiv 0 , 4 \pmod 8$, and the union of two irreducible divisors for $d \equiv 2 \pmod 8$. 

The following restriction on the image of the period morphism plays a crucial role in our proof of Theorem~\ref{main-theorem} below. 

\begin{theorem}[{\cite{ogrady2015periods, DebKuz:periodGM}}]
\label{theorem-image-period-morphism} 
The image of the period morphism $p \colon \cN \to \cD$ is contained in the complement $\cD \setminus (\cD_{2} \cup \cD_{4} \cup \cD_{8})$. 
\end{theorem}

\begin{proof}
By Theorem~\ref{theorem-period-factorization}, 
$p \colon \cN \to \cD$ factors through the period morphism for double EPW sextics, 
which by \cite[Theorem 1.3]{ogrady2015periods} has image contained in $\cD \setminus (\cD_{2} \cup \cD_{4} \cup \cD_{8})$.
\end{proof} 

\begin{remark}
Conjecturally, the image of $p \colon \cN \to \cD$ is exactly $\cD \setminus (\cD_{2} \cup \cD_{4} \cup \cD_{8})$. 
Some partial progress on this problem was made in 
\cite{DIM4fold}, where it is shown that $p$ is dominant and 
$p^{-1}(\cD_d)$ is nonempty for $d \geq 10$. 
\end{remark}

Next we give a similar but much easier result restricting the  
periods of ordinary GM surfaces. 
We will use this in one of two proofs given below for Theorem~\ref{main-theorem} in the case of special GM threefolds. 

\begin{lemma} \label{lem:GMK3periods}
Let $W = \Gr(2,5) \cap \bP^6 \cap Q$ be an ordinary GM surface. 
Then $\Pic(W)$ does not contain a
rank two lattice with intersection form given by 
\[
\begin{pmatrix} 10  & 1 \\ 1 & 0 \end{pmatrix}
\quad \text{or} \quad 
\begin{pmatrix} 10  & 3 \\ 3 & 0 \end{pmatrix},
\]
with the first basis vector corresponding to the restriction of the Pl\"ucker polarization.
\end{lemma}
\begin{proof} This is a special case of \cite[Lemma 2.8]{GreerLiTian:PicardMukai}; since a direct proof is short and easy, we give one here.
In either case, the second basis vector is effective, and thus the class of a genus 1 curve $C$. The Pl\"ucker polarization would restrict to a very ample divisor on $C$ of degree 1 or 3, respectively. This is immediately a contradiction in the first case. In the second case, the Pl\"ucker polarization would embed $C$ as a plane cubic curve. However, as the equations of $\Gr(2,5)$ are quadratic, $W$ would contain the entire $\bP^2$ spanned by $C$, a contradiction.
\end{proof}

\subsection{Proof of Theorem~\ref{main-theorem}} 
\label{section-proof-main-theorem}
Let $Y$ be a quartic double solid and $X$ a GM threefold. 
Assume for sake of contradiction that there is an equivalence $\Ku(Y) \simeq \Ku(X)$. 
Combining Proposition~\ref{proposition-enriques-examples}, Lemma~\ref{lemma-K3-cover-equivalence}, and Theorem~\ref{theorem-K3-cover-examples}, 
we obtain an equivalence $\Db(Y_{\br}) \simeq \Ku(X^{\op})$ which is equivariant 
for the residual $\bZ/2$-actions, where $Y_{\br} \subset \bP^3$ is the branch locus of $Y \to \bP^3$ and $X^{\op}$ is the opposite GM variety. 
This induces a $\bZ/2$-equivariant Hodge isometry  
\begin{equation}
\label{theta}
\theta \colon \tH(Y_{\br}, \bZ) \xrightarrow{\, \sim \,} \tH(\Ku(X^{\op}), \bZ) . 
\end{equation} 
To derive a contradiction, we will use a description of the 
$\bZ/2$-invariants on each lattice. 
The case where $X$ is ordinary so that $X^{\op}$ is a fourfold was 
already addressed in Proposition~\ref{proposition-tH-A12}. 
Note that $Y_{\br}$, as well as $X^{\op}$ when $X$ is special, is a K3 surface, 
so its Mukai lattice up to sign is just the full integral cohomology (Example~\ref{example-tH-K3}). 

\begin{lemma} \label{lem:Z2invariantlattice}
The invariant sublattice $\tH(Y_{\br}, \bZ)^{\bZ/2}$, as well as 
$\tH(X^{\op}, \bZ)^{\bZ/2}$ when $X$ is a special GM threefold, are isomorphic to $A_1^{\oplus 2}$ and given as follows.
\begin{enumerate}
\item \label{item:Yeigenspace} 
$\tH(Y_{\br}, \bZ)^{\bZ/2} = \llangle (1, -A, 1), (1, 0, -1) \rrangle$, where 
$A \in \Pic(Y_{\br})$ is the degree $4$ polarization. 

\item \label{item:K3eigenspace} 
$\tH(X^{\op}, \bZ)^{\bZ/2} = \llangle (1, -B, 4), (2, -B, 2) \rrangle$, 
where $B \in \Pic(X^{\op})$ is the degree $10$ Pl\"{u}cker polarization. 
\end{enumerate}
\end{lemma}

\begin{proof}
Recall that a spherical twist $\rT_\cE$ acts on cohomology by the reflection $\rho_{v(\cE)}$, defined by
\begin{equation} \label{eq:rhovE}
\rho_{v(\cE)} (v) = v + v(\cE) \cdot \bigl(v(\cE), v\bigr),
\end{equation}
while tensoring with a line bundle $\cL$ acts by multiplication with $\ch(\cL)$.
In case \eqref{item:Yeigenspace}, this shows
\begin{equation} \label{eq:PhiYbraction}
(\Phi_{Y_{\br}})_*  (1, -A, 1)  = (1, 0, -1) \quad \text{and} \quad
(\Phi_{Y_{\br}})_* (1, 0, -1)  = -(1, -A, 1).
\end{equation}
Hence the two vectors are preserved by the action of $\Phi_{Y_{\br}}^2[-1]$, i.e.~by Theorem~\ref{theorem-K3-cover-examples}\eqref{quartic-involution} they are $\bZ/2$-invariant. In case \eqref{item:K3eigenspace}, we instead apply Theorem~\ref{theorem-K3-cover-examples}\eqref{K3-cover-GM}; using $v(\cU_{X^\op}) = (2, -B, 3)$ it is a straightforward computation that both classes are invariant under $\Phi_{X^\op}[-1]$. One also sees immediately that $\bZ/2$ acts by $-1$ on the orthogonal complements of the sublattices 
\begin{align} 
\label{LYbr}
L_{Y_{\br}} & = \rH^0(Y_{\br}, \bZ) \oplus A \bZ \oplus \rH^4(Y_{\br}, \bZ) \subset \tH(Y_{\br}, \bZ)  ,  \\ 
\label{LXop}
L_{X^{\op}} & = \rH^0(X^{\op}, \bZ) \oplus B \bZ \oplus \rH^4(X^{\op}, \bZ) \subset \tH(X^{\op}, \bZ). 
\end{align} 
As $\bZ/2$ evidently does not act as the identity on these rank $3$ sublattices, this leaves the claimed rank $2$ lattices as the only possibility for the invariants. 
\end{proof}

Now we can finish the proof of Theorem~\ref{main-theorem}. 
We break the proof into two cases, depending on whether $X$ is ordinary or special. 

\begin{proof}[Assume $X$ is ordinary]
The primitive sublattice $L_{Y_{\br}} \subset \tH(Y_{\br}, \bZ)$ from~\eqref{LYbr} has  signature $(2,1)$, contains 
the $\bZ/2$-invariant sublattice, 
consists of Hodge classes, and has discriminant $-4$. 
Therefore, its image $K' \subset \tH(\Ku(X^{\op}), \bZ)$ under the $\bZ/2$-equivariant 
Hodge isometry~\eqref{theta} has the same properties. 
By Proposition~\ref{proposition-tH-A12} the $\bZ/2$-invariants of $\tH(\Ku(X^{\op}), \bZ)$ is the canonical $A_1^{\oplus 2}$ sublattice. 
So by Proposition~\ref{proposition-tH}, $K'$ corresponds to a rank $3$ primitive positive definite sublattice $K \subset \rH^4(X^{\op}, \bZ)$ which contains $\rH^4(\Gr(2,5),\bZ)$, consists of Hodge classes, and has discriminant~$4$. 
This means the period of $X^{\op}$ lies in the divisor $\cD_{4}$, which contradicts Theorem \ref{theorem-image-period-morphism}.
\end{proof}

\medskip \noindent 
\emph{Assume $X$ is special.} 
We give two proofs. 
The first is shorter but relies on a hard result from \cite{categorical-cones} to reduce to the ordinary case, while the second uses only the easy Lemma~\ref{lem:GMK3periods} and is the starting point for our proof of Theorem \ref{main-theorem-2} in \S\ref{section-deformation-equivalence}. 

\begin{proof}[Proof 1]
The description of generalized partners from \cite[Lemma 3.8]{GM-derived} shows that there exists an ordinary GM threefold $X'$ which is a generalized partner of $X$. 
Then by \cite[Theorem 1.6]{categorical-cones} we have an equivalence $\Ku(X) \simeq \Ku(X')$, so we are reduced to the ordinary case handled above. 
\end{proof} 

\begin{proof}[Proof 2] 
The rank $3$ lattices $L_{Y_{\br}}$ and $L_{X^{\op}}$ from~\eqref{LYbr} and~\eqref{LXop} are given explicitly by 
\[
\begin{pmatrix}
2 & 0 & -1 \\
0 & 2  & -1 \\
-1 & -1 & 0
\end{pmatrix}
\quad \text{and} \quad
\begin{pmatrix}
2 & 0 & -1 \\
0 & 2  & -2 \\
-1 & -2 & 0
\end{pmatrix},
\]
where the first two basis vectors are the $\bZ/2$-invariant ones from Lemma \ref{lem:Z2invariantlattice}, and the third is the class of a point. 
As the isometry $\theta$ in~\eqref{theta} is $\bZ/2$-equivariant, 
it sends the basis of the invariant lattice $A_1^{\oplus 2} \subset L_{Y_{\br}}$ to the basis of $A_1^{\oplus 2} \subset L_{X^{\op}}$, up to sign and permutation. By \eqref{eq:PhiYbraction}, one can  precompose with an appropriate power of  $\Phi_{Y_{\br}}$ to ensure it sends the first vector to the first vector. As $\theta$ respects the standard orientation of positive definite $2$-planes by \cite[Corollary 4.10]{HMS:Orientation}, it  also sends the second basis vector to the second basis vector.

Thus, the rank $4$ lattice in $\tH(X^{\op},\bZ)$ generated by $\theta(L_{Y_{\br}})$ and $L_{X^{\op}}$ has intersection form 
\[ \begin{pmatrix}
2 & 0 & -1 & -1 \\
0 & 2 & -2 & -1 \\
-1 & -2 & 0 & -r \\
-1 & -1 & -r & 0
\end{pmatrix}
\]
for some $r \in \bZ$. The discriminant of this lattice, which has to be positive, is
$-4r^2 - 12 r + 1$, and hence the only possibilities are $r = -3, -2, -1, 0$. The fourth basis vector is a class of the form $(r, D, s)$ with 
\[ D^2 = 2rs, \quad -s - D.B - 4r = -1 \quad -2s - D.B -2r = -1. 
\]
Solving for $s$ shows that $B$ and $D$ span a lattice with intersection form
\[ \begin{pmatrix}
10 & 1-6r \\
1-6r & 4r^2 
\end{pmatrix}.
\]
For $r = 0$ this immediately contradicts Lemma \ref{lem:GMK3periods}. For $r = -1$, we get
$(B-D)^2 = 0$ and $B.(B-D) = 3$, contradicting the second case of Lemma \ref{lem:GMK3periods}. Similarly, for $r = -2$ we have $(D-B)^2 = 0$ and $B.(D-B) = 3$, and for $r = -3$ we have $(2B-D)^2 = 0$  and $B.(2B-D) = 1$, in all cases contradicting  Lemma \ref{lem:GMK3periods}.
\end{proof}

\begin{remark} By \cite[Proposition 5.23]{IHC-CY2} the Kuznetsov component $\Ku(Y)$ of a quartic double solid determines its intermediate Jacobian $J(Y)$, and 
by the classical Torelli theorem proved in \cite{Debarre:doublequarticTorelli}, $J(Y)$ determines $Y$. As a corollary, categorical Torelli holds for quartic double solids: $\Ku(Y)$ determines $Y$.

The methods used in the second proof for the case where $X$ is special above also lead to a short direct proof of this categorical Torelli statement, which we sketch briefly. Given an equivalence $\Ku(Y_1) \simeq \Ku(Y_2)$, one considers the associated $\bZ/2$-equivariant Hodge isometry 
\[ \theta\colon \tH\left(Y_{1, \br}, \bZ \right) \xrightarrow{\, \sim \,}  \tH\left(Y_{2, \br}, \bZ \right).\]
If $\theta $ sends the class of a point to a class in the lattice $L_{Y_2, \br}$, then up to composition with the involution induced by $\Phi_{Y_2, \br}$ it sends it to the class of a point; in this case, an easy argument shows that $\theta$ induces a Hodge isometry $\rH^2(Y_{1, \br}, \bZ) \xrightarrow{\, \sim \,} \rH^2(Y_{2, \br}, \bZ)$ preserving the polarizations. By the Torelli theorem for K3 surfaces there is a polarized isomorphism $Y_{1, \br} \to Y_{2, \br}$, and thus an isomorphism $Y_1 \cong Y_2$. Otherwise, one considers again the rank $4$ lattice spanned by $L_{Y_2, \br}$ and the image of the class of a point; a computation exactly as in the proof above shows that $\rH^2({Y_{2, \br}}, \bZ)$ contains a square zero class of degree $2$, which cannot exist on a quartic K3.
\end{remark}

%%%%%%%%%%%%%%%%%%%%%%%%%%%%%%%%%%%%%%%%%%%%%%%%%%%%%%

\section{Deformation equivalence} 
\label{section-deformation-equivalence}

We continue to work over the complex numbers. 
In this section, we prove  Theorem \ref{main-theorem-2}, which says that Kuznetsov components of quartic double solids and of GM threefolds are deformation equivalent. The idea is based on the second proof of Theorem~\ref{main-theorem} in the case where the GM threefold is special. Namely, in the case $r = -1$ we used the fact that there is no GM K3 surface with Picard lattice $\begin{pmatrix} 10 & 3 \\ 3 & 0 \end{pmatrix}$. However, there is such a polarized K3 surface $S$ --- indeed, as we will see it is given by a very general quartic containing a line. 

This directly suggest our strategy: we extend the $\bZ/2$-action on GM K3 surfaces, given by Theorem \ref{theorem-K3-cover-examples}\eqref{GM-involution}, to one acting also on $S$, in such a way that it becomes conjugate with the $\bZ/2$-action on $S$ as a quartic K3 surface given by Theorem \ref{theorem-K3-cover-examples}\eqref{quartic-involution}. 
To construct this as an enhanced $\bZ/2$-action in a family, we apply the general results from  \S\ref{section-group-actions-cats}. Taking the associated $\bZ/2$-invariant category then proves Theorem~\ref{main-theorem-2}.

\subsection{Quartics containing a line}

From now on, let $S \subset \bP^3$ be a quartic K3 surface that is very general among quartics containing a line $L \subset S$; in particular, $S$ has Picard rank two.  The projection from the line induces an elliptic fibration $S \to \bP^1$; let $E$ be the class of a fiber, a plane cubic.  
The Picard lattice of $S$ is given by $\begin{pmatrix} 4 & 3 \\ 3 & 0 \end{pmatrix}$, with respect to the basis given by $D$, the class of the hyperplane section, and $E$. Both $E$ and the class $L = D-E$ of the line are classes of curves that can be contracted, and so generate extremal rays of the Mori cone. A simple computation then confirms the following result.
 
\begin{lemma}
The Mori cone of $S$ is given by $\langle E, D-E\rangle$, and the nef cone by 
$\langle E, 3D - E \rangle$. 
\end{lemma}
In particular, $H := D +E$ is a polarization of degree 10. Moreover, with respect to the basis $H, E$, the intersection matrix becomes $\begin{pmatrix} 10 & 3 \\ 3 & 0 \end{pmatrix}$, one of the possibilities we had to exclude for GM K3s in Lemma \ref{lem:GMK3periods} for the second proof of Theorem~\ref{main-theorem}. As the general degree 10 K3 surface is obtained as a GM surface \cite{Mukai:newdevelopment}, $S$ is a degeneration of GM K3 surfaces.

\subsection{Stable rank $2$ bundles on $S$}
In our arguments, we will sometimes have to show that rank $2$ bundles on $S$ that we  construct are stable. In all cases, this will follow from the following classical lemma due to Mukai.

\begin{lemma}[{\cite[Corollary 2.8]{Mukai:BundlesK3}}] \label{lemma:Mukai}
Let $0\to \cF \to \cE \to \cG \to 0$ be a short exact sequence of sheaves on a K3 surface with $\Hom(\cF, \cG) = 0$. Then
\[\dim \Ext^1(\cF, \cF) + \dim \Ext^1(\cG, \cG) \le \dim \Ext^1(\cE, \cE).
\]
\end{lemma}

Recall from Example~\ref{example-tH-K3} the Mukai Hodge structure $\tH(S, \bZ)$ associated to the K3 surface~$S$. The Mukai vector of an object $\cE \in \Db(S)$ is defined by 
\[ v(\cE) = \ch(\cE) \cdot \sqrt{\td(S)} \in  \tH^{1,1}(S, \bZ) = \rH^0(S, \bZ) \oplus \rH^{1, 1}(S, \bZ) \oplus \rH^4(S, \bZ) . 
\]
The Mukai pairing is defined by
$$\bigl((r, c, s), (r', c', s') \bigr) = cc' - rs' -r's$$ 
and satisfies $- \chi(\cE, \cF) = \left( v(\cE), v(\cF)\right)$. 
Given a polarization $A$, we define $\mu_A$-(semi)stability via the slope function
\[ \mu_A(\cE) = \frac{A.c_1(\cE)}{A^2 \rk(\cE)}.\]

\begin{proposition}\label{prop:stablerk2bundles}
Let $S$ be a very general quartic K3 surface containing a line, and let $A$ be a polarization on $S$. 
\begin{enumerate}
    \item \label{item:sphericalstable}
    Let $\cV$ be a rank $2$ spherical vector bundle on $S$. Then $\cV$ is $\mu_A$-stable. 
    \item \label{item:semirigidstable}
    Let $\cV$ be a rank $2$ vector bundle on $S$ such that $\Hom(\cV, \cV) = \bC$ and $\Ext^1(\cV, \cV) = \bC^2$. Assume that $c_1(\cV) - D$ is divisible by $2$. 
    Then $\cV$ is $\mu_A$-stable  unless it is destabilized by $\cO(B)$ for
    \[ 
    B = \frac{c_1(\cV) \pm (D-2E)}2. 
    \]
    If $A.(D - 2E) \neq 0$ (e.g. if $A = D$) and $\cV$ is unstable, then more precisely the destabilizing object is $\cO(B)$ where the sign is chosen such that $\mu_A(\cO(B)) > \mu_A(c_1(\cV))$; 
    moreover, in this case 
    the spherical twist $\rT_{\cO(B)}\cV$ is a $\mu_A$-stable vector bundle of the same Mukai vector. 
\end{enumerate}
\end{proposition}

\begin{proof}
If $\cV$ is not stable, then there is a short exact sequence 
\[ 0 \to \cL_1 \to \cV \to \cL_2 \otimes I_Z \to 0
\]
where $\cL_1, \cL_2$ are line bundles with $\mu_A(\cL_1) \ge \mu_A(\cL_2)$, and $Z$ is a zero-dimensional subscheme of $S$. Since $\Hom(\cV, \cV) = \bC$ we have $\cL_1 \neq \cL_2$, and hence $\Hom(\cL_1, \cL_2) = 0$.  Mukai's Lemma, Lemma \ref{lemma:Mukai}, shows that $Z$ is empty in case \eqref{item:sphericalstable}, and that $Z$ is either empty or a single point in case \eqref{item:semirigidstable}. We write $C_i = c_1(L_i)$ for $i = 1, 2$, and so $v(\cL_i) = \bigl(1, C_i, \frac 12 C_i^2 + 1\bigr)$. 

From 
\[ -2 = v(\cL_1)^2,  \quad -2 + 2\cdot \mathrm{length}(Z) = v(\cL_2 \otimes I_Z)^2 \quad \text{and} \quad  v(\cV)^2 = \bigl(v(\cL_1) + v(\cL_2 \otimes I_Z)\bigr)^2
\]
we obtain
\begin{align} 4 + v(\cV)^2 - 2\cdot \mathrm{length}(Z) &= 2\bigl(v(\cL_1),  v(\cL_2 \otimes I_Z)\bigr) = 2 C_1 C_2 - C_1^2 - 2 - C_2^2 -2 + 2\cdot\mathrm{length}(Z)
\label{eq:C1-C2} \nonumber \\ 
(C_1 - C_2)^2 & = -8 + 4 \cdot \mathrm{length}(Z) - v(\cV)^2. 
\end{align}
We write $C_1 - C_2 = dD + eE$, which gives $(C_1 - C_2)^2 = 4d^2 + 6de$.

In case \eqref{item:sphericalstable} we have $v(\cV)^2 = -2$ and $\mathrm{length}(Z) = 0$, and thus \eqref{eq:C1-C2} becomes $4d^2 + 6de = -6$. Hence $d$ is divisible by 3. But then $4d^2 + 6de$ is divisible by $9$ while $-6$ is not, a contradiction. 
%which leads to a contradiction modulo 9.

Now consider case \eqref{item:semirigidstable}, where $v(\cV)^2 = 0$. When $Z$ is a point, we get $ 4d^2 + 6de = -4$, a contradiction modulo 3.
Thus $Z$ has to be empty, in which case \eqref{eq:C1-C2} is equivalent to $4d^2 + 6de = -8$, i.e.~$d(2d+3e) = -4$. 
 The possibilities $d = \pm 1, \pm2, \pm 4$ lead to
\[
C_1 - C_2 = \pm (D- 2E), \pm (2D-2E), \pm (4D-3E),
\]
respectively. Since we assume that $c_1(\cV)-D$ is divisible by two, the first pair of solutions is the only one with $C_1$ and $C_2$ integral. This proves the first part of the claim. 

Finally, if $\cV$ is not stable, we have seen that its Harder-Narasimhan filtration is
\[ 0 \to \cO(B) \to \cV \to \cO(c_1(\cV)-B) \to 0 \]
with $\Hom(\cO(B), \cO(c_1(\cV)-B)) = 0$ and hence $\Hom(\cO(B), \cV) = \bC$.
Since $\cV$ is simple, $\Ext^2(\cO(B), \cV) = \Hom(\cV, \cO(B))^\vee = 0$. Using \eqref{eq:C1-C2} once more we obtain 
\[ \chi(\cO(B), \cV) = - v(\cO(B))^2 - \bigl(v(\cO(B)), v(\cO(c_1(\cV)-B))\bigr) = 2 - 2 = 0.
\]
Thus
$\Hom^\bullet(\cO(B), \cV) = \bC \oplus \bC[-1]$, which gives a $4$-term short exact sequence
\[
0 \to \cO(B) \to \cV \to \rT_{\cO(B)} \cV \to \cO(B) \to 0.
\]
This shows that $\rT_{\cO(B)} \cV $ is a rank $2$ vector bundle with $\Hom(\cO(B), \rT_{\cO(B)} \cV ) = 0$. Since $\rT_{\cO(B)}$ is an equivalence, it is also a simple vector bundle with $\Ext^1(\rT_{\cO(B)} \cV , \rT_{\cO(B)} \cV ) = \bC^2$. 
Applying the previous results of the proposition shows that  $\rT_{\cO(B)} \cV$ is stable as claimed. 
\end{proof}

Let $\cU$ be the $\mu_D$-stable spherical vector bundle  with $v(\cU) = (2, -D-E, 3)$, whose existence and uniqueness were first proved in \cite[Theorem 2.1]{Kuleshov-existence} and \cite[Corollary 3.5]{Mukai:BundlesK3}, respectively.
Then by Proposition~\ref{prop:stablerk2bundles}, $\cU$ is stable with respect to any polarization.  Moreover: 

\begin{lemma} \label{lem:Ufiberstable}
Let $\cU$ be the slope-stable spherical vector bundle with $v(\cU) = (2, -D-E, 3)$. Then the restriction of $\cU$ to any fiber of the elliptic fibration induced by $E$ is stable.
\end{lemma}
\begin{proof}
By the previous observation, $\cU$ is stable with respect to the ample polarization $E + \epsilon D$ for $\epsilon > 0$, and thus at least semistable with respect to the nef polarization $E$. As $E.c_1(\cU) = -3$ is odd, this means it is stable with respect to $E$ (in the sense that any saturated subsheaf of $\cU$ has strictly smaller slope).

Now we can follow the proof of \cite[Theorem 5.2]{Lang:ssposchar}.
Assume that there is a curve $C \subset S$ of class $E$ such that $\cU_C$ is unstable. Then there exists a line bundle $L_C$ on $C$ of degree $d \le -2$ and a surjection $\cU_C \twoheadrightarrow L_C$. Let $\cK$ be the kernel of the composition $\cU \twoheadrightarrow \cU_C \twoheadrightarrow L_C$. Since $c_1(\cK).E = c_1(\cU).E$, it is also slope-stable with respect to $E$. Since $v(L_C) = (0, E, d)$ we have $v(\cK) = (2, -D-2E, 3 - d)$ and hence
\[
v(\cK)^2 = (D + 2E)^2 - 2\cdot 2\cdot (3-d) = 16 - 12 + 4d \le -4,
\]
a contradiction.
\end{proof}

\begin{remark}
As pointed out by Kuznetsov, the bundle $\cU$ can be described explicitly as follows. 
Note that $\Hom(\cO(D), \cO_L(2)) = \bC^2$ and the corresponding map $\cO(D)^{\oplus 2} \to \cO_L(2)$ is surjective. 
Let $\cF$ be the vector bundle defined by the short exact sequence 
\begin{equation*}
    0 \to \cF \to \cO(D)^{\oplus 2} \to \cO_L(2) \to 0. 
\end{equation*}
A computation shows that $\Hom(\cF, \cF) = \bC$ and 
$v(\cF^{\svee}) = (2, - D-E, 3)$. Therefore, $\cF^{\svee}$ is a rank $2$ spherical vector bundle, and hence by Proposition~\ref{prop:stablerk2bundles}.\ref{item:sphericalstable} it is $\mu_D$-stable. 
We conclude that $\cU \cong \cF^{\svee}$. 
\end{remark}

\subsection{Conjugate autoequivalences}
Since we can consider $S$ either as a quartic K3 surface, or as a degeneration of a GM K3 surface, there are two natural autoequivalences associated to it by Theorem~\ref{theorem-K3-cover-examples}: 
\begin{align*}
\Phi^{\mathrm{quartic}} &= \left(\rT_\cO  \circ\bigl( \blank \otimes \cO(D) \bigr)\right)^2[-1]  \\
\Phi^{\mathrm{GM}} & = \rT_\cU \circ \rT_\cO \circ \bigl(\blank \otimes \cO(D+E) \bigr)[-1], 
\end{align*}
where $\rT_\cO$ and $\rT_{\cU}$ are the spherical twists around $\cO$ and $\cU$. 

We will show that $\rT_{\cO(-D)}^2 \Phi^{\mathrm{GM}}$ and $\Phi^{\mathrm{quartic}}$ are conjugate to each other. Since by Theorem~\ref{theorem-K3-cover-examples}\eqref{K3-cover-GM} the latter generates the residual $\bZ/2$-action on $\Db(S)$ as the CY2 cover of the Kuznetsov component of the associated quartic double solid, it will follow that the former also generates a $\bZ/2$-action on $\Db(S)$.

The following lemma will allow us to prove our identity by computing the images of skyscraper sheaves of points.
\begin{lemma} \label{lem:howtoshowidentityinDbS}
Let $S$ be a smooth projective K3 surface, and let $F_1, F_2$ be two autoequivalences of $\Db(S)$. Assume that:
\begin{enumerate}
    \item \label{item:Hactionsame} $F_1$ and $F_2$ have the same action on $\tH(S, \bZ)$.
    \item \label{item:skyscraperssame} Applying $F_1$ and $F_2$ to skyscraper sheaves of points gives the same set of objects:
    $$\{F_1(\cO_s)\}_{s \in S(\bC)} = \{F_2(\cO_s)\}_{s \in S(\bC)}.$$
\end{enumerate}
Then $F_1$ and $F_2$ are isomorphic functors. 
\end{lemma}
\begin{proof}
Assumption \eqref{item:skyscraperssame} implies that $F_1^{-1} \circ F_2$ sends skyscraper sheaves of points to skyscraper sheaves of points. By \cite[Corollary 5.23]{Huybrechts:FM}, it is the composition of the  pushforward along an automorphism $f$ of $S$ and tensoring with a line bundle $M$. By \eqref{item:Hactionsame}, $F_1^{-1} \circ F_2$ acts trivially on cohomology. So  $F_1^{-1} \circ F_2$ preserves the class $v(\cO_S)$ and $M$ is trivial. But then the action of $f_*$ on $\rH^2(S, \bZ)$ is trivial, and thus $f$ is the identity by the Torelli theorem for K3 surfaces.
\end{proof}

\begin{remark} \label{rem:bijectionfromstability}
Condition \eqref{item:skyscraperssame} holds automatically when the $F_i(\cO_s)$ are, up to the same shift, slope-stable vector bundles for the same polarization. Indeed, both $F_1$ and $F_2$ induce an injective map from $S$ to the moduli space of vector bundles of class $F_i(\cO_s)$; since this moduli space is two-dimensional and irreducible, both maps are bijections on closed points. 
\end{remark}

We first prove that the actions of $\Phi^{\mathrm{quartic}}$ and $\Phi^\mathrm{GM}$ on cohomology are conjugate.

\begin{lemma} \label{lem:sameHaction}
Let $\Psi \colon \Db(S) \to \Db(S)$ be the autoequivalence given by 
\[ \Psi(\blank) = \rT_{\cO(-D)} \bigl(\blank \otimes \cO(-E) \bigr) . \]
Then $\Phi^{\mathrm{quartic}}$ and 
$\Psi^{-1} \circ \rT_{\cO(-D)}^2 \Phi^{\mathrm{GM}} \circ \Psi$ have the same action on $\tH(S, \bZ)$. 
\end{lemma}
\begin{proof}
The lemma follows by direct computation, similar to the proof of Lemma \ref{lem:Z2invariantlattice}.
More precisely, we find that $\Phi^{\mathrm{quartic}}$ and
$\Psi^{-1} \circ \rT_{\cO(-D)}^2 \Phi^{\mathrm{GM}} \circ \Psi$ both act on a basis of the algebraic part of $\tH(S, \bZ)$ as follows:  
\begin{align*}
  (1,0,0) & \mapsto (-1,D,-2) , \\
  (0,D,0) & \mapsto (-4,3D,-4) , \\
  (0,E,0) & \mapsto (-3,3D-E,-3) , \\
  (0,0,1) & \mapsto (-2,D,-1) .
\end{align*}
Moreover, they each act by multiplication by $-1$ on the orthogonal complement of $D$ and $E$ in $\rH^2(S, \bZ)$, and thus they agree on all of $\tH(S, \bZ)$. 
\end{proof}

\begin{proposition} \label{prop:conjugateZ2subgroups}
The two autoequivalences $\Phi^{\mathrm{quartic}}$ and 
$\Psi^{-1} \circ \rT_{\cO(-D)}^2 \Phi^{\mathrm{GM}} \circ \Psi $ of $\Db(S)$ are isomorphic functors.
\end{proposition}

\begin{proof}
Since $\Psi^{-1}(\blank) = \left(\blank \otimes \cO(E) \right) \circ \rT_{\cO(-D)}^{-1} $, the claim is equivalent to
\begin{equation} \label{autoeqequationtoprove}
    \left(\rT_\cO  \circ\bigl( \blank \otimes \cO(D) \bigr)\right)^2
    = \bigl(\blank \otimes \cO(E) \bigr) \circ \rT_{\cO(-D)} \circ 
    \rT_\cU \circ \rT_\cO \circ \bigl(\blank \otimes \cO(D+E) \bigr) 
   \circ \rT_{\cO(-D)} \circ \bigl(\blank \otimes \cO(-E) \bigr).
\end{equation}
By Lemma~\ref{lem:sameHaction}, the two sides have the same action on $\tH(S, \bZ)$. By Lemma~\ref{lem:howtoshowidentityinDbS}  and Remark~\ref{rem:bijectionfromstability}, it will thus be enough to show that applying the LHS and the RHS to skyscraper sheaves of points yields slope-stable vector bundles.

We first consider the LHS of \eqref{autoeqequationtoprove} applied to the skyscraper sheaf at $s \in S(\bC)$. We note that $\rT_{\cO}(\cO_s\otimes \cO(D)) = \rT_{\cO}(\cO_s) = I_s[1]$. Since $\cO(D)$ is very ample, the sheaf $I_s(D)$ has three sections, no higher cohomology, and is globally generated; thus the image of $\cO_s$ is $\cF_s[2]$, where $\cF_s$ is defined by the short exact sequence
\[
0 \to \cF_s \to \cO \otimes \rH^0(I_s(D)) \to I_s(D) \to 0.
\]
It is a simple rank $2$ bundle with $\Ext^1(\cF_s, \cF_s) = \Ext^1(\cO_s, \cO_s) = \bC^2$ and 
$c_1(\cF_s) = -D$. 
 By Proposition~\ref{prop:stablerk2bundles}, $\cF_s$ is $\mu_D$-slope stable unless it is destabilized by 
\[
\cO\left(\frac{-D - (D-2E)}2\right) = \cO(-L), 
\]
where we recall that $L = D-E$ is the class of the line on $S$. Since the natural morphism
$$ \Hom(\cO(-L), \cO \otimes H^0(I_s(D)) = H^0(I_s(D)) \to
\Hom(\cO(-L), I_s(D)) =  H^0(I_s(D+L)),$$
induced by multiplying with the defining section of $\cO(L)$ is injective, $\Hom(\cO(-L), \cF_s) = 0$, and thus $\cF_s$ is $\mu_D$-stable.

Now we consider the RHS of \eqref{autoeqequationtoprove}, applied to $\cO_s$. After the first three steps we reach
$$
\bigl( \rT_{\cO(-D)} (\cO_s\otimes \cO(-E))\bigr) 
    \otimes \cO(D+E) = I_s(E)[1].$$
Since $\cO(E)$ is globally generated and has two sections (inducing the elliptic fibration), there is a unique section of $I_s(E)$ vanishing at the elliptic fiber $E_s$ containing $s$; thus 
\[
\rT_\cO(I_s(E)[1]) = I_{s/E_s}[1],
\]
where $I_{s/E_s}$ denotes the image of the composition $I_s \to \cO_S \to \cO_{E_s}$. 
By Lemma \ref{lem:Ufiberstable}, $\cU|_{E_s}$ is a stable vector bundle on the elliptic curve $E_s$ of rank two and degree $-3$. By Serre duality, \[ 
\Ext_S^1(\cU, I_{s/E_s}) = \Ext_{E_s}^1(\cU|_{E_s}, I_{s/E_s}) = \Hom_{E_s}(I_{s/E_s}, \cU|_{E_s})^\vee = 0, \]
and therefore by Riemann--Roch $\Hom(\cU, I_{s/E_s}) = \bC$. 
Using stability of $\cU|_{E_s}$ once more, we see that this map must be surjective. Therefore,
\[
\rT_{\cU}\bigl(I_{s/E_s}[1]) = \cV_s[2]
\]
where the vector bundle $\cV_s$ with $v(\cV_s) = (2, -D - 2E, 4)$ is defined by the short exact sequence
\begin{equation} \label{eq:defVs}
    0 \to \cV_s \to \cU \to I_{s/E_s} \to 0.
\end{equation} 
By Proposition \ref{prop:stablerk2bundles}, $\cV_s$ is $\mu_D$-stable unless it is destabilized by
\[
\cO \left(  \frac {-D -2E - (D - 2E)} 2\right)  = \cO(-D).
\]
We claim that $\Hom(\cO(-D), \cV_s[i]) = 0$ for all $i$ if $s \notin L$, and $\Hom(\cO(-D), \cV_s) \neq 0$ if $s \in L$.

To prove the claim, first note that by stability
$\Ext^2(\cO(-D), \cU) = \Hom(\cU, \cO(-D))^{\vee} = 0$. Moreover, 
$\Ext^1(\cO(-D), \cU) = \Ext^1(\cU, \cO(-D))^\vee = 0$; otherwise, the corresponding extension would define $\mu_D$-stable vector bundle of Mukai vector $(3, -2D-E, 6)$, 
whose square is $-8$, a contradiction. 
Therefore, 
\[
\Hom^\bullet(\cO(-D), \cU) = \bC^2[0] = \Hom^\bullet(\cO(-D), I_{s/E_s}).
\]
To prove the claim, we need to show that the long exact sequence obtained by applying $\Hom(\cO(-D), \blank)$ to the map $\cU \to I_{s/E_s}$ in \eqref{eq:defVs} induces an isomorphism in degree 0 if and only if $s \notin L$. 

Every nonzero morphism $\cO(-D) \to \cU$ fits into short exact sequence
\[ 
0 \to \cO(-D) \to \cU \to I_{s'}(-E) \to 0
\]
An easy computation shows that $s' \in L$, as otherwise \begin{equation*}
\Ext^1(I_{s'}(-E), \cO(-D)) = H^1(I_{s'}(L))^\vee = 0.
\end{equation*}
Conversely, if $s' \in L$, the extension exists, and thus arises from a morphism $\cO(-D) \to \cU$. Now observe that the composition 
$\cO(-D) \to \cU \to I_{s/E_s}$ is nonzero unless there is a morphism $I_{s'}(-E) \to I_{s/E_s}$, which exists if and only if $s = s'$. This proves the claim. 

Thus $\rT_{\cO(-D)}\cV_s = \cV_s$ if $s \notin L$, and Proposition \ref{prop:stablerk2bundles} shows that $\rT_{\cO(-D)}\cV_s$ is $\mu_D$-stable for all $s \in S$. Thus both the LHS and the RHS of \eqref{autoeqequationtoprove} send $\cO_s$ to the shift by $[2]$ of a $\mu_D$-stable vector bundle.
By Lemma \ref{lem:howtoshowidentityinDbS} and Remark \ref{rem:bijectionfromstability}, this completes the proof of the proposition.
\end{proof}

\subsection{$\bZ/2$-action in families}
\begin{proof}[Proof of Theorem~\ref{main-theorem-2}]
Consider the quasi-projective moduli space $\cF_6$ of degree 10 (genus~6) polarized K3 surfaces. It contains as a Noether-Lefschetz divisor the locus of K3 surfaces lattice polarizable with the lattice $\begin{pmatrix} 4 & 3 \\ 3 & 0 \end{pmatrix}$ spanned by $D$ and $E$ where $D+E$ corresponds to the given polarization of degree 10.

Let $C \subset \cF_6$ be a smooth curve intersecting this divisor transversely at a single point $o \in C$, such that $o$ corresponds to a quartic K3 of Picard rank two containing a line. By base change to a finite cover if necessary, we can assume that there exists a family of K3 surfaces $\pi \colon \cS \to C$ with polarization $\cH$. The very general point of $C$ necessarily corresponds to a K3 surface of Picard rank one, which is a GM surface by \cite{Mukai:newdevelopment} and \cite[Lemma 2.8]{GreerLiTian:PicardMukai}; shrinking $C$ if necessary, we can assume $C \setminus \{o\}$ parameterizes only GM surfaces.

Up to possibly passing to a cover of $C$, there exists a rank $2$ vector bundle $\fU$ on $\cS$ whose restrictions to fibers $\cS_c$ is the unique $\cH$-stable vector bundle of Mukai vector $(2, -\cH_c, 3)$; so it is the tautological subbundle on GM fibers, and  the unique stable bundle of Mukai vector  $(2, -D-E, 3)$ appearing in Lemma~\ref{lem:Ufiberstable} on $\cS_o$. 
Let $i_o\colon \cS_o \to \cS$ be the inclusion of the special fiber. 
Then the object $i_{o*} \cO(-D)$ is spherical;  indeed, by our choice of the curve $C \subset \cF_6$ the object $\cO(-D)$ does not deform in the family $\pi \colon \cS \to C$, so the claim holds by \cite[Proposition~1.4]{HT:Pobjects}. 
Now consider the following  autoequivalence of $\Dperf(\cS)$:
\begin{equation*}
\Pi  = \rT_{i_{o*} \cO(-D)} \circ \rT_{\fU/C} \circ \rT_{\cO/C} \circ \bigl(\blank \otimes \cO_{\cS}(\cH) \bigr)[-1],
\end{equation*}
where $\rT_{i_{o*} \cO(-D)}$ is the spherical twist around  $i_{o*} \cO(-D)$, 
and $\rT_{\fU/C}$ and $\rT_{\cO/C}$ are the spherical twists associated to the spherical functors 
$\Dperf(C) \to \Dperf(\cS)$ given by $F \mapsto \pi^*F \otimes \fU$ and $F \mapsto \pi^*F$.  Each of these three spherical twists are associated to $C$-linear spherical functors: in the case of $\rT_{{i_{o*}} \cO(-D)}$ for the functor $\Dperf(o) \to \Dperf(\cS)$, $V \mapsto i_{o*} \cO(-D) \otimes V$. Hence the  spherical twists, and thus $\Pi$, are also $C$-linear.

The autoequivalence $\Pi$ induces an autoequivalence $\Pi_c$ on $\Dperf(\cS_c)$ for every fiber by base change. 
For $c \neq 0$, it is the residual action $\Phi^{\mathrm{GM}}$ of Theorem~\ref{theorem-K3-cover-examples} on GM K3s as the CY2 covers of Kuznetsov components of corresponding special GM threefolds. For $c = o$, by 
\cite[Proposition~2.7 and Proposition~2.9]{HT:Pobjects} we get
\[ \Pi_o = \rT_{\cO(-D)}^2 \circ \rT_\cU \circ \rT_\cO \circ \bigl(\blank \otimes \cO(D+E)\bigr)[-1] = 
\rT_{\cO(-D)}^2  \Phi^{\mathrm{GM}},
\]
the autoequivalence considered in Lemma~\ref{lem:sameHaction} and Proposition~\ref{prop:conjugateZ2subgroups}.
In particular, for every $c \in C$, the autoequivalence $\Pi_c$ is an involution that generates a $\bZ/2$-action on $\Dperf(\cS_c)$. 

The functor $\Pi \circ \Pi$ is  the identity on every fiber, and thus sends every skyscraper sheaf $\cO_s$ for $s \in \cS$ to itself. Since $\Pi \circ \Pi$ is a Fourier--Mukai transform by construction, one can easily adapt the proof of \cite[Corollary~5.23]{Huybrechts:FM} to show that it is given by tensor product with a line bundle; this line bundle is trivial on the fibers of $\pi$, i.e.~ it is pulled back from $C$. Shrinking $C$ further if necessary, we may assume this line bundle to be trivial, and therefore that $\Pi$ is an involution of $\Dperf(\cS)$. 

We therefore have a homomorphism $\phi \colon \bZ/2 \to \pi_0(\Aut(\Dperf(\cS)/C))$ as in Corollary \ref{corollary-obstruction-HH}, and want to show that the obstruction to the existence of an $\infty$-lift vanishes.  The restriction of $\phi$ to $\cS_o$ can be lifted to a $\bZ/2$-action, as by Proposition~\ref{prop:conjugateZ2subgroups} it is conjugate to the residual action on $\Dperf(\cS_o)$ coming from its realization as the CY2 cover of the Kuznetsov component of the corresponding quartic double solid. By Proposition \ref{proposition-etale-vanishing}, we can replace $C$ by an \'etale neighborhood $B$ of $o \in C$ such that the obstruction to lifting $\phi$ vanishes.

We have thus obtained a $\bZ/2$-action on $\Dperf(\cS)$ over $B$ whose generator acts by the involution $\Pi$. The associated invariant category $\cC \coloneqq \Dperf(\cS)^{\bZ/2}$ has the properties claimed in Theorem~\ref{main-theorem-2}. 
Indeed, by Proposition~\ref{proposition-invariants-smooth-proper} the category $\cC$ is smooth and proper over $B$, and by Lemma~\ref{lemma-invariants-bc} the fiber $\cC_b$ over $b \in B$ is given by the invariant category $\Db(\cS_b)^{\bZ/2}$ for the induced $\bZ/2$-action $\phi_b$. 
Let $Y \to \bP^3$ be the quartic double solid branched along $\cS_o$, and for $b \neq 0$ let $X_b$ be the GM threefold opposite (in the sense of Definition~\ref{definition-opposite}) to the GM surface $\cS_b$. 
By construction, for $b = o$ the
action of $\phi_o$ on $\Db(\cS_o)$ is conjugate to the residual action on $\Db(\cS_o)$ from Theorem~\ref{theorem-K3-cover-examples}\eqref{quartic-involution}, and thus its invariant category is equivalent to $\Ku(Y)$ (Lemma~\ref{lemma-reconstruction}). 
Similarly, by construction and Remark~\ref{remark-connected-enriques}, for $b \neq o$ the 
action $\phi_b$ is equivalent to the residual action on  $\Db(\cS_b)$ from  Theorem~\ref{theorem-K3-cover-examples}\eqref{K3-cover-GM}, so its invariant category is equivalent to $\Ku(X_b)$. 
\end{proof}

%%%%%%%%%%%%%%%%%%%%%%%%%%%%%%%%%%%%%%%%%%%%%%%%%%%%%%

\newcommand{\etalchar}[1]{$^{#1}$}
\providecommand{\bysame}{\leavevmode\hbox to3em{\hrulefill}\thinspace}
\providecommand{\MR}{\relax\ifhmode\unskip\space\fi MR }
\providecommand{\MRhref}[2]{%
  \href{http://www.ams.org/mathscinet-getitem?mr=#1}{#2}
}
\providecommand{\href}[2]{#2}

%%%%%%%%%%%%%%%%%%%%%%%%%%%%%%%%%%%%%%%%%%%%%%%%%%%%%%

\end{document}